\documentclass[10pt]{article}
\usepackage{geometry}                
\usepackage{amsmath}
\usepackage{amsthm}
\usepackage[parfill]{parskip}    
\usepackage{graphicx}
\usepackage{amssymb}
\usepackage{epigraph}
\usepackage{bibentry}
\usepackage{epstopdf}
\usepackage{dsfont}
\usepackage{url}
\usepackage{array}
\usepackage{color}
\usepackage{tikz}
\usepackage{float}
\usepackage{graphicx}
\usepackage{caption}
\usepackage{subcaption}
\usepackage{algorithm}
\usepackage{algorithmicx}
\usepackage{algpseudocode}
\usepackage{multirow}
\usepackage{authblk}

\theoremstyle{definition}

\theoremstyle{plain}
\newtheorem{theorem}{Theorem}
\newtheorem{proposition}{Proposition}
\newtheorem{lemma}{Lemma}

\theoremstyle{definition}

\theoremstyle{remark}

\newtheorem{remark}{Remark}

\newcommand{\Tcal}{\mathcal{T}}

\newcommand{\Scal}{\mathcal{S}}

\newcommand{\Zbb}{\mathbb{Z}}

\newcommand{\Ucal}{\mathcal{U}}

\newcommand{\Rbb}{\mathbb{R}}

\newcommand{\Rcal}{\mathcal{R}}
\newcommand{\Jcal}{\mathcal{J}}

\newcommand{\diag}{\mathbf{diag}}

\newcommand{\ext}{\mathbf{ext}}

\DeclareMathOperator{\Rot}{\mathbf{Rota}}
\newcommand{\Fold}{\mathbf{Fold}}
\DeclareMathOperator{\ArcTan}{\mathbf{atan}}

\newcommand{\matQi}{{Q^{(i)}}}

\newcommand{\vecciT}{{c^{(i)T}}}
\newcommand{\scldi}{{d^{(i)}}}

\nobibliography*

\title{Compact Disjunctive Approximations to Nonconvex Quadratically Constrained Programs}
\author[*]{Hongbo Dong}
\author[*]{Yunqi Luo}
\affil[*]{Department of Mathematics and Statistics, Washington State University, Pullman WA 99164}
\date{November 19, 2018}                                           

\begin{document}
\maketitle

\begin{abstract}
Decades of advances in mixed-integer linear programming (MILP) and recent development in 
mixed-integer second-order-cone programming (MISOCP) have translated very mildly to progresses in 
global solving nonconvex mixed-integer quadratically constrained programs (MIQCP).
In this paper we propose a new approach, namely Compact Disjunctive Approximation (CDA), to
approximate nonconvex MIQCP to arbitrary precision by \textit{convex} MIQCPs, which can be solved 
by MISOCP solvers. For nonconvex MIQCP with $n$ variables and $m$ general quadratic constraints, 
our method yields relaxations with at most $O(n\log(1/\varepsilon))$ number of continuous/binary variables 
and linear constraints, together with $m$ \textit{convex} quadratic constraints, where $\varepsilon$ is the approximation accuracy.
The main novelty of our method lies in a very compact lifted mixed-integer formulation for approximating the (scalar) square function.
This is derived by first embedding the square function into the boundary of a three-dimensional second-order cone,
and then exploiting rotational symmetry in a similar way as in the construction of BenTal-Nemirovski approximation.
We further show that this lifted formulation characterize the union of finite number of simple convex sets, which naturally relax the square
function in a piecewise manner with properly placed knots. 
We implement (with JuMP) a simple adaptive refinement algorithm. Numerical 
experiments on synthetic instances used in the literature show that our prototypical implementation (with hundreds of lines of Julia code) can 
already close a significant portion of gap left by various state-of-the-art global solvers on more difficult instances, indicating strong 
promises of our proposed approach.
\end{abstract}

\section{Introduction and Summary of Contributions}
In this paper we propose a new approach towards globally solving nonconvex mixed-integer quadratically constrained programming (MIQCP) in the following form
\begin{equation}\label{MIQCP} 
\begin{aligned}
\min_{x,\zeta} & \quad f^{(0)}(x, \zeta) := x^T Q^{(0)} x + c^{(0)T} x +d^{(0)T}\zeta \\
s.t.& \quad f^{(i)}(x, \zeta) := x^T \matQi x + \vecciT x + d^{(i)T} \zeta - f^{(i)} \leq 0, \quad i=1,...,m, \\
& \quad x_j \in [\ell_j, u_j], \ \forall j = 1,...,n, \quad \zeta \in \{0,1\}^s.\\
\end{aligned}\tag{MIQCP}
\end{equation}
For any $i=0,...,m$, $Q^{(i)}$ is an $n\times n$ real symmetric matrix with (potentially) both positive and negative eigenvalues, $c^{(i)}$ and $d^{(i)}$ are 
vectors in the Euclidean space $\Rbb^n$ and $\Rbb^s$, respectively. We further assumed that all continuous variables are bounds, 
i.e., $\ell_j >-\infty$ and $u_j < +\infty$.

\ref{MIQCP} is a very expressive problem class. The Stone-Weierstrass Theorem states that any continuous
function on a closed and bounded region of $\Rbb^n$ can be approximated arbitrarily close with a polynomial function, which can be further reformulated as a quadratic
system with additional variables and quadratic constraints. Nonconvex quadratic constraints also naturally arise in many areas of science and engineering. 
For example, the AC optimal power flow (ACOPF) is a long-standing and fundamental problem in power system optimization. Its rectangular form is a nonconvex QCP 
with continuous variables, e.g., \cite{CainOneillCastillo}. See  \cite{GhaddarMarecekMevissen,Kocuk_Strong_2016,BynumCastilloWatsonLaird2018} for some 
recent development in global optimization methods. Optimization problem such as unit commitment and optimal transmission switching often include additional
integer variables and ACOPF as part of the problem structure, e.g., see  \cite{FuShahidehpourLi2005,BurakDeySun2017}.
Observing the importance of advancing optimization algorithms for MIQCP, in recent years a specialized instance library named QPLIB \cite{QPLIB} (\url{http://qplib.zib.de/})
has been developed to hosts a collection challenging instances from various application areas for benchmarking.

During the last two decades several solvers have been developed to solve the more general problem class of mixed-integer nonlinear programs (MINLP) to global 
optimality. The most well-known examples include BARON \cite{baron17.8.9}, ANTIGONE \cite{ANTIGONE2014},  Couenne\cite{couenne,CouennePaper}, Lindo API \cite{Lindo2009}
and SCIP\cite{GleixnerEtal2018ZR,SCIP-MIQCP-2012,SCIP-2018}. They can be used to solve general MIQCPs.
All of these solvers are based on spatial branch-and-bound with different implementation and choices in cut generation,
branching, bound tightening and domain propagation, etc. However one common and crucial ingredient among all solvers is 
that convex relaxations are usually constructed in a term-wise manner, i.e., for each pair of $(i,j)$ such that the nonlinear term $x_i x_j$ 
exists in the problem, an additional continuous variable, denoted by $X_{ij}$, is introduced and constrained by McCormick inequalities 
(or RLT inequalities)
\begin{equation}\label{RLT}\begin{aligned}
\ell_i x_i + \ell_j x_j - \ell_i \ell_j &\leq X_{ij} \leq 
\ell_i x_j + u_j x_i - \ell_i u_j \\
u_i x_i + u_j x_j - u_i u_j &\leq X_{ij} \leq 
\ell_j x_i + u_i x_j - u_i \ell_j,
\end{aligned}\end{equation}
or related improvements (e.g., edge concave relaxations \cite{Misener_Global_2012}, multi-term cuts \cite{BaoSahinidisTawarmalani2009}, etc.). 
One main difficulty of this term-wise approach is that the problem is lifted into a much higher dimensional space (especially when there are many nonlinear terms),
and the RLT inequalities are usually weak. For effective branch-and-bound, one essential challenge is to derive relaxations with balanced strength and computational
complexity.

Another line of research is to derive strong convex relaxations or even complete convexification by imposing conic constraints. 
The Shor relaxation \cite{Sho87,PolReWol95} is a standard way of deriving semidefinite programming (SDP) relaxations for MIQCPs by lifting to 
the matrix space where the quadratic form $\begin{pmatrix}1\\x\end{pmatrix} 
\begin{pmatrix}1\\x\end{pmatrix}^T$ lies in.
It was observed by Anstreicher \cite{kurt10_QCQP} that the combination of RLT inequalities and the positive 
semidefinite (PSD) constraint often provides much tighter convex (semidefinite) relaxations than each of these approaches alone. This phenomenon is also related to the discovery that 
a large class of nonconvex quadratic program can be equivalently formulated as linear programs over the completely positive cone \cite{Bur09}, as the intersection
of RLT and PSD constraints is related to the doubly nonnegative relaxation \cite{AnsBur07}, which is known to be tight in low dimensions. 
General quadratically constrained programs can also be reformulated as with generalized notion of complete-positivity \cite{Burer2011,PVZ2015}. 
Despite successes in some special cases, e.g., some combinatorial optimization problems 
\cite{RendlRinaldiWiegele10,KrislockMalickRoupin2013,Krislock:2014}, nonconvex quadratic program with linear constraints \cite{BurerChen2013}, in general it is difficult 
to exploit such strong conic relaxations within a branch-and-bound framework due to the lack of stable and scalable SDP algorithms. 
Some attempts were made to project strong relaxations in the lifted space back to the original variable space as cutting planes and cutting surfaces to avoid this problem
\cite{SBL08c,DongNonconvexQP}. There has also been much efforts in developing global solution strategies by exploiting specific problem structure.
For example, see \cite{GhaddarMarecekMevissen,Kocuk_Strong_2016,Kocuk2018,BynumCastilloWatsonLaird2018} for some recent development on the ACOPF problem. 

An important subclass of (\ref{MIQCP}) is when all quadratic forms $Q^{(i)}$ ($i=0,...,m$) are positive semidefinite. We call such problems convex MIQCPs 
(although they are still nonconvex problems).
Recent years have witnessed much interests and progresses in mixed-integer second-order cone programming (MISOCP) (e.g., see \cite{Hande_MISOCP}),
to which convex MIQCP can be reformulated. In leading solvers  MISOCPs are either solved by direct branch-and-bound (by applying interior point methods
to their continuous relaxations) 
or outer approximation (OA) algorithms \cite{DuranGrossmann1986,LeyfferPhDThesis}, where a sequence of mixed-integer linear programs (MILP) are solved. We mention 
an important way of constructing polyhedral approximations to the second-order cone is proposed by Ben-Tal and Nemirovski in \cite{BN01}, which exploits
 rotational symmetry. This approximation is first applied to MISOCP in \cite{Vielma08} and further studied in \cite{KrokhmalVinel2014}. 

Despite significant computational advancement of MILP in the last a few decades and MISOCP in recent years,
it is reasonable to say that such progresses have translated very mildly into progresses in solving general MIQCPs.
In this paper we propose an approach to solve general MIQCPs with moderately larger convex MIQCPs,
and empirically show that our prototypical implementation
can already close a large portion of gap left by leading global solvers on some more difficult instances.  We now summarize our 
proposed approach in the rest of this section.

The first step is to convexify all nonconvex quadratic constraints by  diagonal perturbation. That is, 
for all $i=0,...,m$ such that $Q^{(i)}$ is not positive semidefinite, to compute vector $\delta^{(i)} \in \Rbb^n$ such that 
$\matQi + \diag(\delta^{(i)})$ is. This can be done by letting  $\diag(\delta^{(i)}) = \tau \cdot I$ where $I$ is the identity matrix and 
$\tau$ is the absolute value of the most negative eigenvalue of $\matQi$. Alternatively, as described in Section \ref{sec:adaptive_refine},
one can achieve this by solving structured SDPs.
By introducing auxiliary variables $y_j = x_j^2$, (\ref{MIQCP}) is equivalently written as
\begin{equation}\label{MIQCP-ref}
\begin{aligned}
\min_{x \in \Rbb^n, \ \zeta\in \{0,1\}^s} & \quad  x^T \left(Q^{(0)} + \diag(\delta^{(0)}) \right) x + c^{(0)T} x +d^{(0)T} \zeta - \delta^{(0)T} y \\
s.t.& \quad x^T \left(\matQi + \diag(\delta^{(i)})\right) x + \vecciT x + d^{(i)T} \zeta \leq f^{(i)} + \delta^{(i)T} y, \quad i=1,...,m, \\
& \quad (x_j, y_j) \in \Scal_{[\ell_j, u_j]}, 
\end{aligned}
\end{equation}
where 
\[
\Scal_{[\ell, u]} := \left\{ (x,y) \in \Rbb^2 \ \middle| \  y=x^2, x \in [\ell,u]\right\}.
\]
Note that all of the nonconvexity in the continuous variables are now ``packed" into sets $\left\{\Scal_{[\ell_j, u_j]}\right\}_{j=1}^n$. 
This reformulation has already been adopted in work of Saxena, Bonami and Lee \cite{SBL08c,SBL08b}.

The main novelty of our paper is to develop two \textit{compact disjunctive approximation} sets to $\Scal_{[\ell,u]}$, denoted by 
$D_{\nu}(\ell,u)$ and $D^+_{\nu}(\ell,u)$, where $\nu$ is some positive integer controlling the approximation accuracy. 
As $\nu \mapsto +\infty$, both approximation sets converge to $\Scal_{[\ell,u]}$ uniformly. These approximation sets admit 
integer formulations that are very economical. By replacing $\Scal_{[\ell_j, u_j]}$ in (\ref{MIQCP-ref}) for all $j$ with 
$D_{\nu}(\ell_j,u_j)$ or $D^+_{\nu}(\ell_j,u_j)$ and properly 
chosen $\nu$, we can approximately solve general (\ref{MIQCP}) by solving moderately larger convex MIQCPs.


The rest of our paper proceeds as follows. In Section \ref{construct} we present the main construction of our approximation
sets $D_\nu(\ell, u)$ and $D^+_\nu(\ell, u)$ by embedding into a three-dimensional second-order cone and exploiting 
rotational symmetry. In Section \ref{sec:mip} we present integer 
formulations for these two approximation sets. We further show in Section \ref{sec:disjunctive} that they characterize the union of finitely many 
simple sets naturally approximate the square function in a piecewise manner. In Section \ref{sec:adaptive_refine}
and \ref{sec:comp} we describe a simple adaptive refinement algorithm and empirically show the promises of our approach with 
numerical results. We conclude the paper with discussions and future work in Section \ref{sec:conclude}. Throughout our paper $\Rbb$ is the 
set of real numbers and $\Zbb$ is the set of all integers. $\|\bullet\|$ denotes the Euclidean norm in $\Rbb^n$ unless stated otherwise.

%

\section{Construction of Approximation Sets}\label{construct}
We start by a simple observation that $y = x^2$ can be equivalently written as
\begin{align}
x^2 + \left(\frac{y-1}{2}\right)^2 & =  \left(\frac{y+1}{2}\right)^2. \label{ineq_concave}
\end{align}
In other words, for any $\ell \leq u$, the linear transformation $\Tcal: \Rbb^2 \mapsto \Rbb^3$ such that 
\[
\Tcal(x,y) = \left(x, \frac{y-1}{2}, \frac{y+1}{2} \right)
\]
defines a bijection between $\Scal_{[\ell,u]}$ and set 
\begin{equation}\label{set:embed}
\left\{(x, v, w) \in [\ell,u ]\times \Rbb^2 \ \middle| \ x^2 + v^2 = w^2, w \geq 0, w - v = 1\right\}.
\end{equation}
Note the set in (\ref{set:embed}) is the intersection of the \textit{boundary} of a three-dimensional second-order cone and an 
affine set. See Figure \ref{fig:embedding} for a graphical illustration where the thick black curve represents the intersection. 
We will now construct approximation sets to $\Scal_{[\ell,u]}$ 
by exploiting this embedding and the rotational symmetry of (the boundary of) the second-order cone. 
Our symmetry-exploiting approximation is inspired by the well-known 
BenTal-Nemirovski approximation \cite{BN01}, which is a lifted polyhedral relaxation of the 
\textit{convex} second-order cone, while we directly address the nonconvex set $\Scal_{[\ell,u]}$.


\begin{figure}[htbp] 
\begin{center}
\includegraphics[width=4in, height=3in]{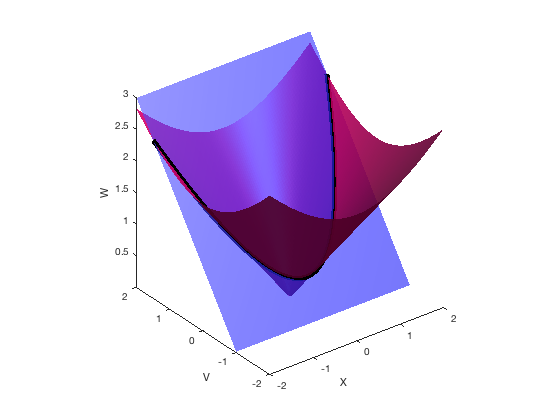}
\caption{Embedding $y=x^2$ into the boundary of second-order cone}
\label{fig:embedding}
\end{center}
\end{figure}


We employ a sequence of ``rotation" and ``folding" transformations.
A transformation, denoted by $\Rot(\bullet; \theta): \Rbb^2 \mapsto \Rbb^2$, that rotates a vector clockwisely by an angle $\theta$ 
is:
\[
\Rot \left([v,w]^T; \theta\right) = \begin{bmatrix}\cos(\theta) & \sin(\theta) \\ -\sin(\theta) & \cos(\theta) \end{bmatrix} 
\begin{bmatrix} v \\ w \end{bmatrix}.
\]
A nonsmooth operation $\Fold(\bullet): \Rbb^2 \mapsto \Rbb\times \Rbb_+$ that ``folds" along the second
 dimension is:
\[
\Fold \left([v,w]^T\right) = \left[ v, |w| \right]^T .
\]
Obviously both operations preserve the 2-norm in $\Rbb^2$, i.e., for any $v, w\in \Rbb$,
\[
\sqrt{v^2 + w^2} = \left\|\Rot \left([v,w]^T; \theta\right)\right\| = \left\|\Fold \left([v,w]^T\right)\right\|, 
\quad\forall \theta.
\]

Let $\arctan: \Rbb \mapsto (-\frac{\pi}{2}, \frac{\pi}{2})$ be the standard inverse tangent function.
A multiple-valued function that maps a nonzero vector in $\Rbb^2$ to its radian angle is denoted by $\ArcTan:\Rbb^2 \setminus \{0\} \mapsto \Rbb$ where 
\begin{equation}\label{def:ArcTan}
\ArcTan([v,w]^T) := \Zbb\cdot 2\pi + \begin{cases}\arctan\left(\frac{w}{v}\right), & \mbox{ if } v > 0, \\
-\frac{\pi}{2},& \mbox{ if } v = 0, w < 0, \\
\arctan\left(\frac{w}{v}\right) - \pi, & \mbox{ if } v < 0, \\
\frac{\pi}{2},& \mbox{ if } v = 0, w > 0. \\
\end{cases} 
\end{equation}
The following lemma is geometrically straightforward and requires no proof.
\begin{lemma}\label{lem:rotflip}
Let $\theta_1$ and $\theta_2$ be angles such that $\theta_1 \leq \theta_2$ and 
$ \theta_2 - \theta_1 \leq 2\pi$.  Define 
$\theta_{mid} :=  (\theta_1 + \theta_2)/2$ and $\theta_d :=  \theta_2 - \theta_1$. 
Let $[v,w]^T$ be a nonzero vector in $\Rbb^2$, the following three conditions are equivalent:
\begin{enumerate}
\item $\ArcTan([v,w]^T) \subseteq [\theta_1, \theta_2] +\Zbb\cdot 2\pi $; 
\item $\ArcTan\left[\Rot \left([v,w]^T; \theta_{mid}\right)\right]  \subseteq [-\theta_d/2, \theta_d/2] + \Zbb\cdot 2\pi$; 
\item $\ArcTan\left[ \Fold\circ \Rot \left([v,w]^T; \theta_{mid}\right)\right] \subseteq [0, \theta_d/2] + \Zbb\cdot 2\pi$.
\end{enumerate}
\end{lemma}

For now on we use $\theta_{min}$ and $\theta_{max}$ to denote the minimal and maximal angles of $[x, \frac{y-1}{2}]^T$ for $(x,y) \in \Scal_{[\ell,u]}$,  i.e.,
\begin{equation}\label{def:theta_minmax}
\theta_{min} := \begin{cases}
\arctan \frac{\ell^2 - 1}{2\ell} & \mbox{ if } \ell > 0 \\
-\frac{\pi}{2} & \mbox{ if } \ell = 0 \\
\arctan \frac{\ell^2 - 1}{2\ell} - \pi & \mbox{ if } \ell < 0 
\end{cases}, \qquad
\theta_{max} := \begin{cases}
\arctan \frac{u^2 - 1}{2 u} & \mbox{ if } u > 0 \\
-\frac{\pi}{2} & \mbox{ if } u = 0 \\
\arctan \frac{u^2 - 1}{2 u} - \pi & \mbox{ if } u < 0 
\end{cases}.
\end{equation}
See Figure \ref{fig:rot1}. 
Note that $\theta_{min}$ and $\theta_{max}$ both take values in $(-\frac{3}{2}\pi, \frac{\pi}{2})$. Such definitions are especially convenient for our purposes as 
the angle $-\frac{3}{2}\pi$ (and $\frac{\pi}{2}$) corresponds to the recession direction of the epigraph of the square function. We will further define 
\begin{equation}\label{def:theta_mid}
\theta_{mid} := \frac{\theta_{max} + \theta_{min}}{2}, \qquad \theta_{d} := \theta_{max} - \theta_{min}.
\end{equation}
For a properly chosen positive integer $\nu$, we then create a collection of lifted variables $\{(\xi_j, \eta_j)\}_{j=1}^\nu$ and 
link them with $(x,y)$ by a sequence of folding and rotations.
\begin{align}
\begin{bmatrix} \xi_1 \\ \eta_1 \end{bmatrix} &= \Fold \circ\Rot\left(\begin{bmatrix} x \\ (y-1)/2 \end{bmatrix}; \theta_{mid}\right), \label{def:xieta1}& \\ 
\begin{bmatrix} \xi_{j+1} \\ \eta_{j+1} \end{bmatrix} &= \Fold \circ \Rot\left(\begin{bmatrix} \xi_j \\ \eta_j \end{bmatrix}; \theta_d/(2^{j+1}) \right), 
 \quad j=1,...,\nu - 1. \label{def:xieta2}
\end{align}
Geometrically, $[\xi_j, \eta_j]$ is a ``copy" of $[x, (y-1)/2]^T$ in the angular sector between $0$ and $\theta_d/(2^j)$ after a sequence of rotation and folding 
operations. Figure \ref{fig:rot1}--\ref{fig:rot3} illustrate the angular sectors (shaded area) and set $\{(x,(y-1)/2 \ | \ (x,y) \in \Scal_{[\ell,u]}\}$ after first two transformations.
\begin{figure}[!htb]
\begin{center}
\begin{subfigure}{.45\textwidth}
  \centering
  \includegraphics[width=0.93\linewidth]{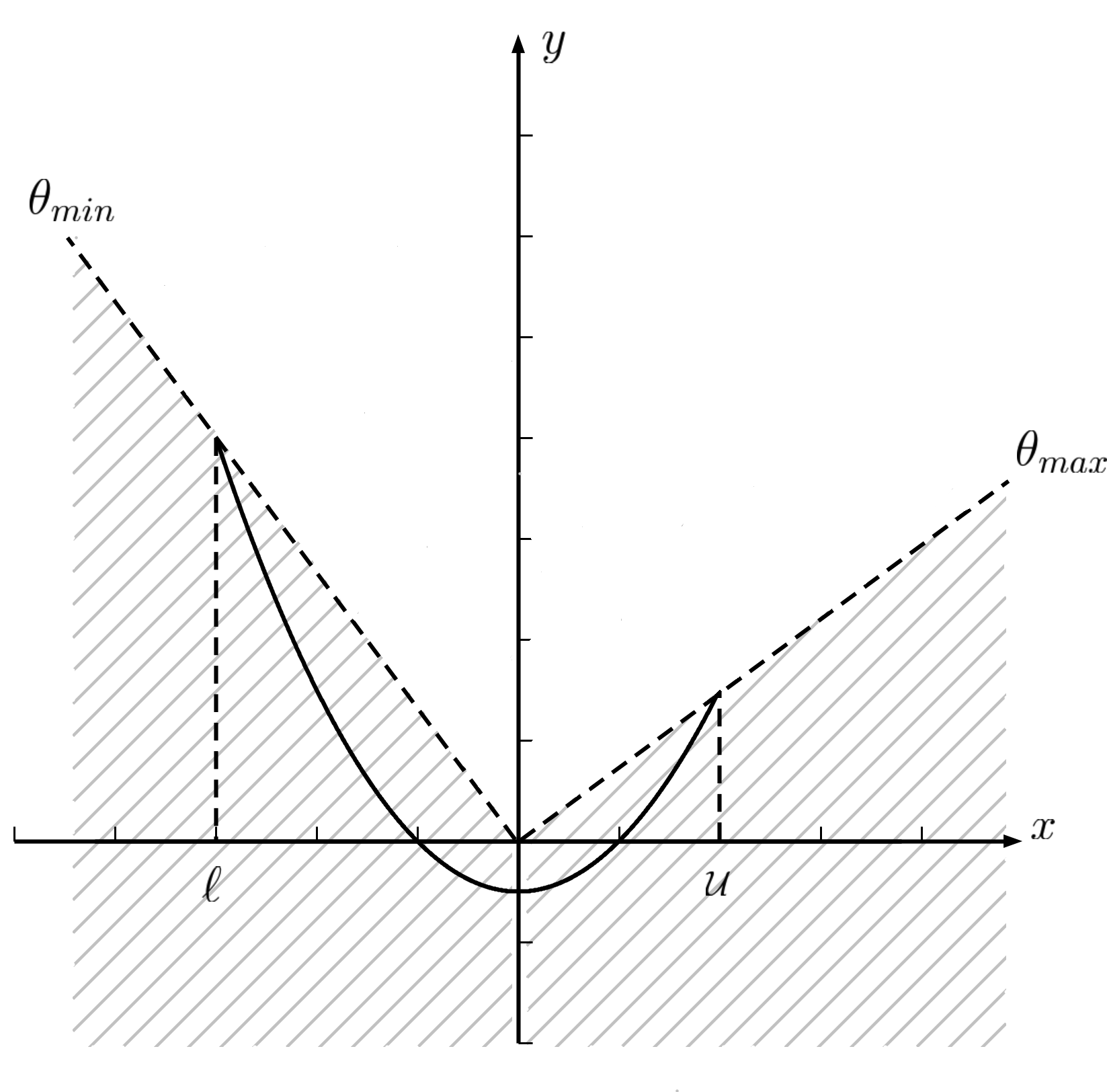}
  \caption{$\theta_{min}$ and $\theta_{max}$ for $(x,\frac{y-1}{2})$ where $(x,y) \in \Scal_{[\ell,u]}$}
  \label{fig:rot1}
\end{subfigure}%
\begin{subfigure}{.45\textwidth}
  \centering
  \includegraphics[scale = 0.15]{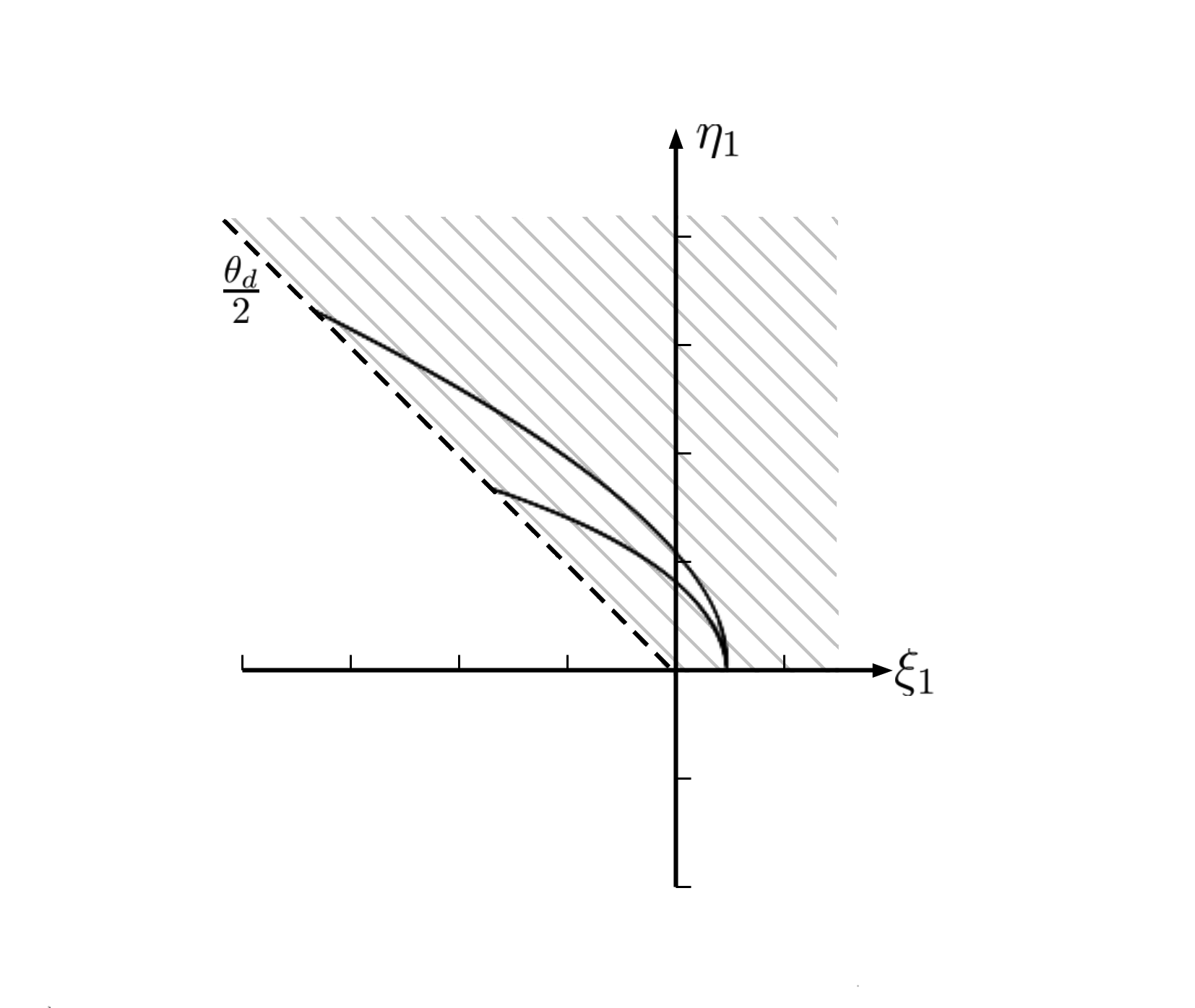}
  \caption{$(\xi_1, \eta_1)$}
  \label{fig:rot2}
\end{subfigure}
\begin{subfigure}{.45\textwidth}
  \centering
  \includegraphics[width=0.88\linewidth]{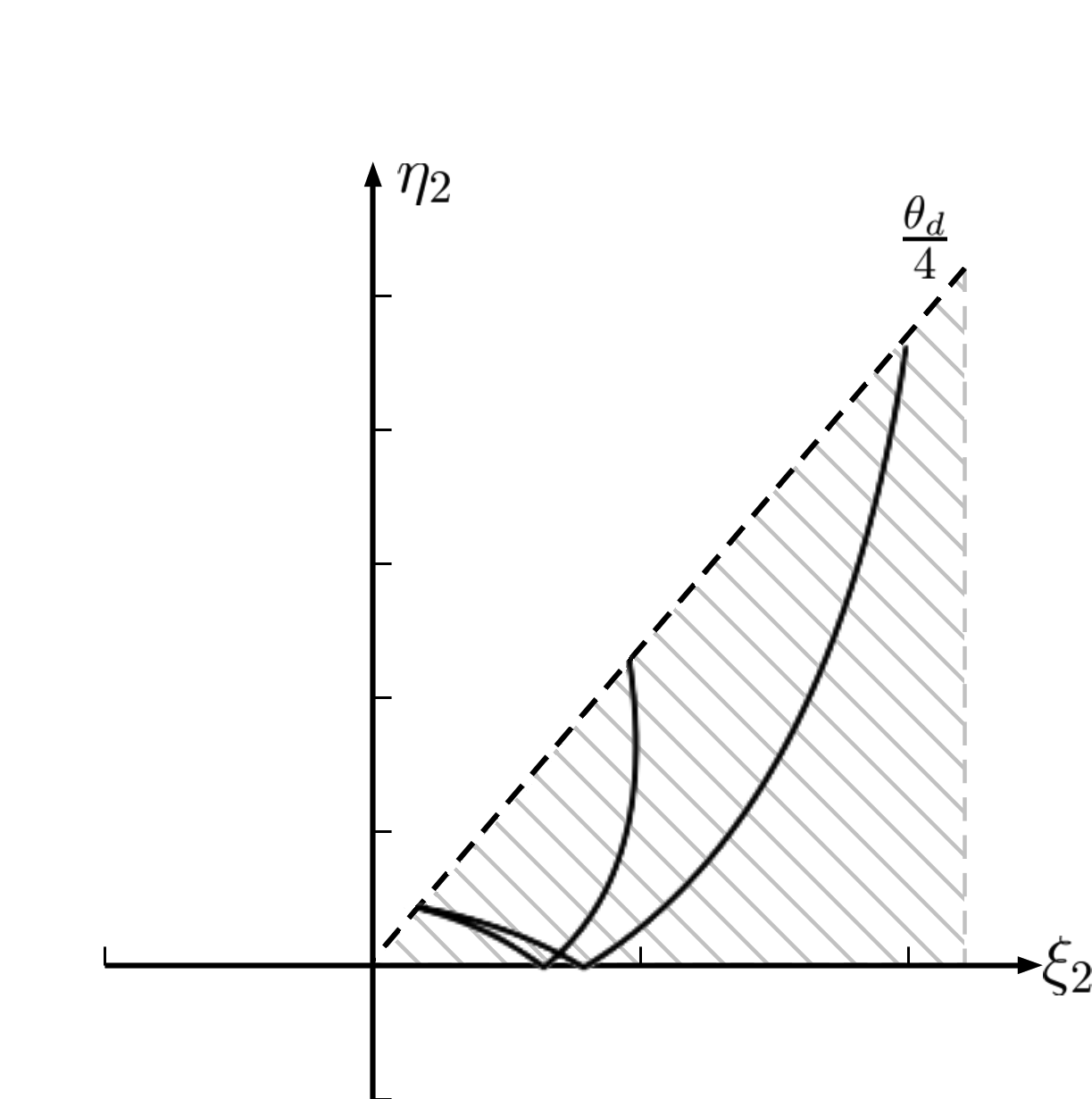}
  \caption{$(\xi_2, \eta_2)$}
  \label{fig:rot3}
\end{subfigure}%
\begin{subfigure}{.45\textwidth}
  \centering
  \includegraphics[width=\linewidth]{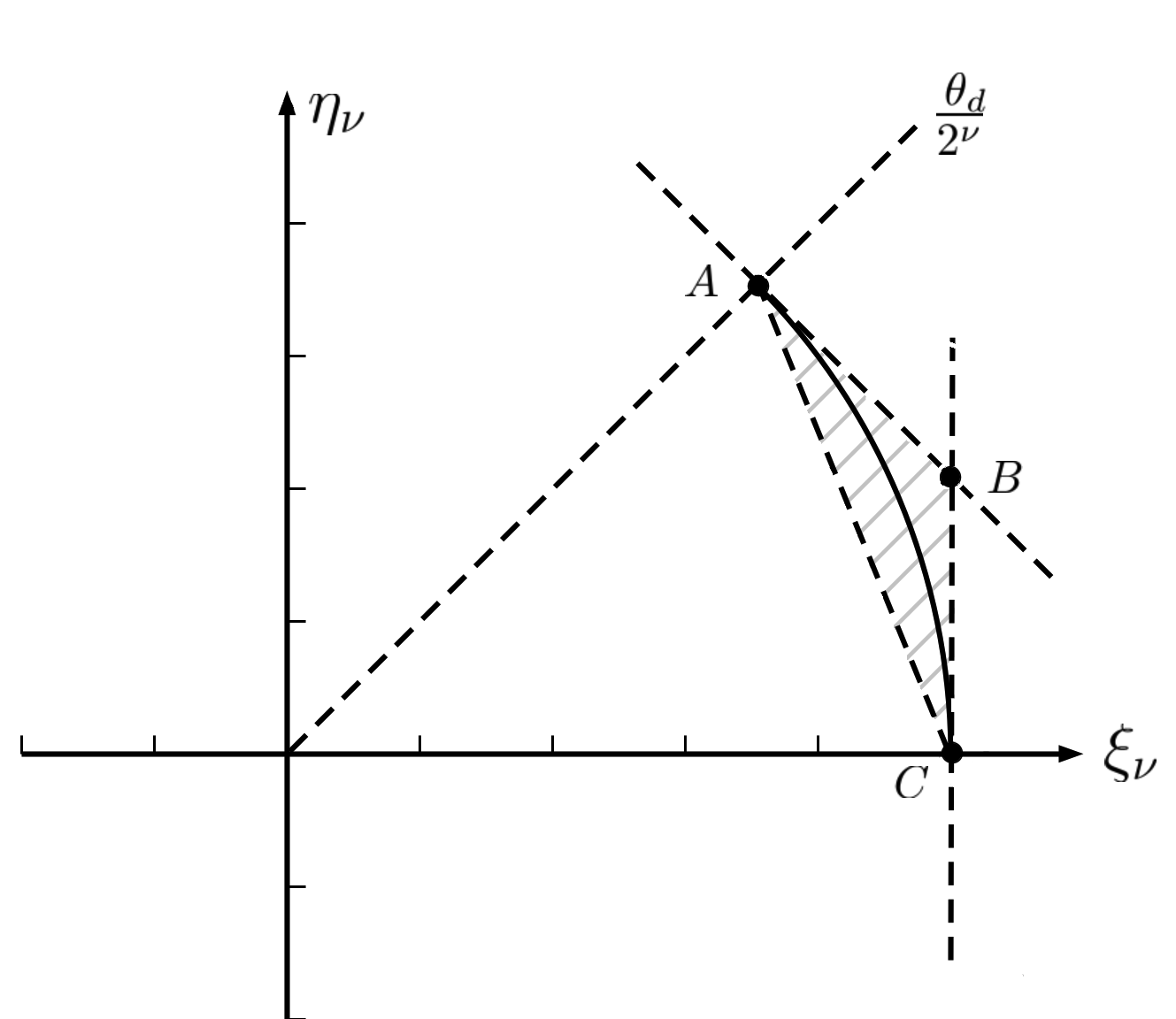}
  \caption{Valid constraints on the triplet $(\xi_\nu, \eta_\nu, y)$.}
  \label{fig:rot4}
\end{subfigure}
\caption{Rotation and folding operations applied to $\{ (x,(y-1)/2) \ | \ (x,y) \in \Scal_{[\ell, u]}\} $ }
\label{fig:rot}
\end{center}
\end{figure}

We then seek to add valid constraints on the last pair of lifted variables $(\xi_\nu, \eta_\nu)$. By the invariance of 
2-norm of $\Fold$ and $\Rot$ it is easy to see that (\ref{def:xieta1}) and (\ref{def:xieta2}) imply
\begin{equation}\label{radius_square}
\xi_\nu^2 + \eta_\nu^2 = x^2 + \left(\frac{y-1}{2}\right)^2.
\end{equation}
If $(x,y) \in \Scal_{[\ell,u]}$, by observation (\ref{ineq_concave}) this quantity should equal $\left(\frac{y+1}{2}\right)^2$. 
Together with the angular restriction of $(\xi_\nu, \eta_\nu)$,  for any fixed $y$, $(\xi_\nu, \eta_\nu)$ must fall on
the arc $\widetilde{AC}$ as in Figure \ref{fig:rot4}, hence lies in the convex triangular region $\Delta ABC$, where 
\[
A =  \left(\frac{y+1}{2}\cos(\theta_d / 2^\nu), \frac{y+1}{2}\sin(\theta_d / 2^\nu)\right), 
\quad B =  \left(\frac{y+1}{2}, \frac{y+1}{2 \cos(\theta_d/(2^{\nu+1}))} \right), \quad C = \left(\frac{y+1}{2}, 0\right).
\]
constructed by two tangent lines and one secant line. This restriction can be described by three valid inequalities
of $(\xi_\nu, \eta_\nu)$ and $y$:
\begin{align}
\xi_\nu \cos\left(\frac{\theta_{d}}{2^{\nu+1}}\right) + \eta_\nu \sin\left(\frac{\theta_{d}}{2^{\nu+1}}\right) &\geq 
\frac{y+1}{2}  \cos\left(\frac{\theta_{d}}{2^{\nu+1}}\right), \label{ineq:secant}\\
\xi_\nu \cos\left(\frac{\theta_{d}}{2^{\nu}}\right) + \eta_\nu \sin\left(\frac{\theta_{d}}{2^{\nu}}\right) &\leq 
\frac{y+1}{2}, \label{ineq:tangent1}\\
\xi_\nu &\leq  \frac{y+1}{2}.\label{ineq:tangent2}
\end{align}
Let $B_{[\ell,u]} \subseteq \Rbb^2$ be a (compact) superset of $\Scal_{[\ell,u]}$ constructed by simple bounds
and the RLT inequalities (\ref{RLT}),
\begin{equation*}
B_{[\ell,u]}:= \left\{(x,y) \in \Rbb^2 \ \middle| \ x \in [\ell,u],  \ 0 \leq y \leq (\ell+u) x - \ell u \right\}. 
\end{equation*}
We define two relaxation sets of $\Scal_{[\ell,u]}$ as follows:
\begin{align*}
D_\nu(\ell, u) & := \left\{ (x, y) \in B_{[\ell, u]} \ \middle| \ 
\begin{aligned}
\exists \{(\xi_j, \eta_j)\}_{j=1}^\nu \mbox{ such that } \\
(\ref{def:xieta1}), \ (\ref{def:xieta2}), \ (\ref{ineq:secant}), \ (\ref{ineq:tangent1}), \ (\ref{ineq:tangent2})
\end{aligned}
\right\}, \\
D^+_\nu(\ell, u) & := \left\{ (x, y) \in B_{[\ell, u]} \ \middle| \ 
\begin{aligned}
\exists \{(\xi_j, \eta_j)\}_{j=1}^\nu \mbox{ such that } \\
(\ref{def:xieta1}), \ (\ref{def:xieta2}), \ (\ref{ineq:secant}), \ y \geq x^2
\end{aligned}
\right\}.
\end{align*}

Note that the definitions of $D_{\nu}(\ell, u)$ and $D^+_{\nu}(\ell, u)$ involve the nonsmooth folding operations,  
hence are not immediately admissible to optimization solvers. We will leave their integer formulations for the next section, while
providing a formal, algebraic proof of the validity of relaxation and their approximation accuracy in the following theorem. 
As $\nu$ increases, both $D_\nu(\ell,u)$ and $D^+_\nu(\ell,u)$ converge to $\Scal_{[\ell,u]}$ very rapidly.
\begin{theorem}\label{thm:approx}
For any $\ell, u$ such that $-\infty < \ell \leq u < +\infty$, we have 
\[\Scal_{[\ell,u]} \subseteq D^+_\nu(\ell,u) \subseteq D_\nu(\ell,u).\]
Suppose that $\nu \geq 2$, then for any $(x,y) \in D_\nu(\ell,u)$, we have the following bounds:
\begin{equation*}
|\sqrt{y}-|x|| \leq \frac{\max(\ell^2,u^2) + 1}{2^{\nu-1}}, \qquad
|y-x^2| \leq \frac{\left[\max(\ell^2,u^2) + 1\right]^2}{2^{2\nu-2}}.
\end{equation*}
Further if $(x,y) \in D^+_\nu(\ell,u)$, then $y \geq x^2$ and the same bounds hold.
\end{theorem}
\begin{proof}
We first show that $\Scal_{[\ell, u]} \subseteq D^+_\nu(\ell, u)$. For any $(x,y) \in \Scal_{[\ell,u]}$,
by $y=x^2$ and the invariance of 2-norm,
\[
(y/2 + 0.5)^2  =  x^2 + (y/2-0. 5)^2 = \xi_j^2 + \eta_j^2 , \quad \forall j=1,...,\nu.
\]
By construction $\ArcTan([x,y/2-0.5]) \in [\theta_{min}, \theta_{max}] + \Zbb\cdot 2\pi$. Iteratively
applying Lemma \ref{lem:rotflip} yields
\[
\ArcTan([\xi_j, \eta_j]^T) \in [0, \theta_d/2^{j}] + \Zbb\cdot 2\pi,\quad j=1,...,\nu.
\]
Take $j=\nu$, then there exists $\tilde{\theta} \in [0, \theta_d/2^{\nu}]$ such that, 
\[
\xi_\nu = (y/2+0.5)\cos \tilde{\theta}, \qquad \eta_\nu = (y/2+0.5)\sin\tilde{\theta}.
\]
Therefore
\[
\xi_\nu \cos\left(\frac{\theta_{d}}{2^{\nu+1}}\right) + \eta_\nu \sin\left(\frac{\theta_{d}}{2^{\nu+1}}\right)
= (y/2+0.5)\cos\left(\frac{\theta_{d}}{2^{\nu+1}}-\tilde{\theta}\right) \geq (y/2+0.5) \cos\left(\frac{\theta_{d}}{2^{\nu+1}}\right),
\] 
where the last inequality is because 
$\tilde{\theta} \in [0, \theta_d/2^{\nu}] \Rightarrow\left| \frac{\theta_{d}}{2^{\nu+1}}-\tilde{\theta} \right| \leq \left|\frac{\theta_{d}}{2^{\nu+1}}\right|$
and $\nu \geq 2 \Rightarrow \theta_d/(2^{\nu+1}) \in [0, \pi/4]$. Therefore $\Scal_{[\ell,u]} \subseteq D_\nu^+(\ell,u)$. 

We now show $D_\nu^+(\ell,u) \subseteq D_\nu(\ell,u)$. Take any $(x,y) \in D_\nu^+(\ell,u)$ with associated $\{\xi_j,\eta_j\}_{j=1}^\nu$. 
Since 
\[
y\geq x^2 \Rightarrow \left(\frac{y+1}{2}\right)^2 \geq \left(\frac{y-1}{2}\right)^2 + x^2 = \xi_j^2 + \eta_j^2, \quad \forall j=1,...,\nu,
\]
there exists $r \in [0, \frac{y+1}{2}]$ and $\tilde{\theta} \in [0,\theta_d/2^\nu]$ such that 
\[
\xi_\nu = r \cos\tilde{\theta}, \mbox{ and } \eta_\nu = r \sin\tilde{\theta}.
\]
It is then further straightforward to verify that 
\[
\xi_\nu  \cos\left(\frac{\theta_{d}}{2^{\nu}}\right)+ \eta_\nu \sin\left(\frac{\theta_{d}}{2^{\nu}}\right) 
= r \cos\left(\frac{\theta_{d}}{2^{\nu}} -\tilde{\theta}\right) \leq r \leq \frac{y+1}{2},
\]
and
\[
\xi_\nu \leq r \leq \frac{y+1}{2}.
\]
So $D_\nu^+(\ell,u) \subseteq D_\nu(\ell,u)$.

Now suppose $(x,y) \in D_\nu{[\ell,u]}$, and $\left\{[\xi_j, \eta_j]^T\right\}_{j=1}^\nu$ be the associated vectors in $D_\nu(\ell,u)$. 
As in Figure \ref{fig:rot4}, the last three inequalities
imply that $[\xi_\nu, \eta_\nu]^T$ is in the triangle formed by the following three points (depending on $y/2+0.5$),
\begin{align*}
\left[ (y/2+0.5) \cos\left(\frac{\theta_{d}}{2^{\nu}}\right), (y/2+0.5) \sin\left(\frac{\theta_{d}}{2^{\nu}}\right) \right]^T, \\
\left[ y/2+0.5 , (y/2+0.5)  \tan\left(\frac{\theta_{d}}{2^{\nu+1}}\right) \right]^T, \\
\left[y/2+0.5, 0\right]^T.
\end{align*}
It is then further straightforward to verify that for any $[\xi_\nu, \eta_\nu]^T$ in this triangle, 
$\sqrt{\xi_\nu^2 + \eta_\nu^2}$ can be bounded by
\[
(y/2+0.5)\cos\left(\frac{\theta_d}{2^{\nu+1}}\right) \leq \sqrt{\xi_\nu^2 + \eta_\nu^2} 
\leq \frac{y/2+0.5}{\cos\left(\theta_{d}/2^{\nu+1}\right)}.  
\]
By the invariance of 2-norm $\xi_\nu^2 + \eta_\nu^2 = x^2 + (y/2-0.5)^2$,
\[
(y/2+0.5)^2\cos^2\left(\frac{\theta_d}{2^{\nu+1}}\right) \leq x^2 + (y/2-0.5)^2 \leq \frac{(y/2+0.5)^2}{\cos^2\left(\theta_{d}/2^{\nu+1}\right)}.
\]
Rearranging terms, we have
\[
y - (y/2+0.5)^2 \sin^2 (\theta_{d}/2^{\nu+1})  \leq x^2  \leq y + (y/2+0.5)^2 \tan^2(\theta_{d}/2^{\nu+1}).
\]
Therefore
\begin{align*}
|y-x^2| & \leq (y/2+0.5)^2 \tan^2(\theta_{d}/2^{\nu+1}) \leq \left[\max(\ell^2,u^2)/2 + 0.5\right]^2 \left(\frac{4}{\pi}\frac{\theta_d}{2^{\nu+1}}\right)^2\\
&\leq \frac{\left[\max(\ell^2,u^2)+1\right]^2 }{2^{2\nu-2}},
\end{align*}
where the second inequality is because $\nu \geq 2$ and $\tan(t) \leq \frac{4}{\pi}t$ for any $t \in [0,\pi/4]$.
Further, note that for any nonnegative $a,b$, $\sqrt{a+b}\leq \sqrt{a}+\sqrt{b}$,
\[
|x| \leq \sqrt{y} + (y/2+0.5) \tan(\theta_{d}/2^{\nu+1}), 
\]
\[
\sqrt{y} \leq |x| + (y/2+0.5) \sin(\theta_{d}/2^{\nu+1}), 
\]
So 
\[
\left|\sqrt{y} - |x|\right| \leq (y/2+0.5) \tan(\theta_{d}/2^{\nu+1}) \leq \left[\max(\ell^2,u^2)/2 + 0.5\right]\cdot \left(\frac{4}{\pi}\frac{\theta_d}{2^{\nu+1}}\right)
\leq \frac{\max(\ell^2,u^2)/2 + 0.5}{2^{\nu-2}}.
\]
The fact that $y\geq x^2$ for all $(x,y) \in D_\nu^+(\ell,u)$ is obvious. 
\end{proof}
By replacing each $\Scal_{[\ell_j, u_j]}$ ($\forall j$) with $D_\nu(\ell_j,u_j)$ we obtain relaxations
for (\ref{MIQCP}), whose optimal value is a valid lower bound to that of (\ref{MIQCP}).
With simple algebra, this relaxation can be written as:
\begin{equation}\label{eq:relax_D}
\begin{aligned}
\tau^{\nu} :=\min_{x, y,\zeta} \quad & f^{(0)}(x,y, \zeta) - \sum_{j=1}^n \delta_j^{(0)}\left(y_j - x_j^2\right) \\
s.t., \quad &  f^{(i)}(x, \zeta) \leq \sum_{j=1}^n \delta_j^{(i)}\left(y_j - x_j^2\right) , \quad \forall i=1,...,m \\
& (x_j, y_j) \in D_\nu(\ell_j, u_j).
\end{aligned}
\end{equation}
where $f^{(i)}, i=0,...,m$ are the quadratic functions in (\ref{MIQCP}). 
By Theorem \ref{thm:approx}, for each $i=1,...,m$, the expression 
$\sum_{j=1}^n \delta_j^{(i)}\left(y_j - x_j^2\right)$ can be bounded by 
\begin{equation}\label{eq:consviol}
\left\| \delta^{(i)} \right\|_1 \frac{\left[\max_{j}{\max(\ell_j^2, u_j^2)}+1\right]^2}{2^{2\nu-2}},\qquad i=1,...,m.
\end{equation}
Therefore if we let $(x(\nu), y(\nu), \zeta(\nu))$ to denote a global optimal solution to (\ref{eq:relax_D}),
then $(x(\nu), \zeta(\nu))$ is an almost-feasible solution to (\ref{MIQCP}) with violation of the $i$-th constraint 
at most (\ref{eq:consviol}). Take $\nu \mapsto 0$ and assume $(x^*, \zeta^*)$ is a limit point of $(x(\nu), \zeta(\nu))$. 
Then by Theorem \ref{thm:approx} and the continuity of $f^{i} \ (i=0,...,m)$, it is straightforward to see 
that $(x^*, \zeta^*)$ is a global optimal solution to (\ref{MIQCP}). Analogous results can be established
for the relaxation with $D^+(\ell_j, u_j)$ ($j=1,...,n$).

In the next section we will show $D_\nu(\bullet, \bullet)$ and $D^+_\nu(\bullet, \bullet)$ admit MILP and 
convex MIQCP formulations, respectively. Therefore related relaxations are computable by global solvers for 
convex MIQCP.

\section{Mixed-Integer Formulations for the Approximation Sets}\label{sec:mip}
As the definitions of $D_\nu(\ell, u)$ and $D^+_\nu(\ell, u)$ use nonsmooth folding operations, 
they are not immediately admissible to most optimization solvers. The problematic constraints are the 
equalities in (\ref{def:xieta1}) and  (\ref{def:xieta2}) 
involving the absolute value functions:
\begin{equation}\label{abs_cons}
\begin{aligned}
\eta_1 &= \left| -\sin(\theta_{mid}) x +  \cos(\theta_{mid})(y-1)/2 \right|, \\
\eta_{j+1} &= \left| -\sin(\theta_{d}/(2^{j+1}) \xi_j +  \cos(\theta_{d}/(2^{j+1})) \eta_j \right|, \quad j=1,...,\nu-1.
\end{aligned}
\end{equation}
Note that the graph of a absolute value function in a bounded interval is simply the union of two line segments.
By disjunctive programming we can model such constraints with additional binary variables \cite{Balas1998}. 
Consider the following set with finite $L$ and $R$,
\[
\{(\eta, \omega) \in \Rbb^2 \ | \ \eta = |\omega|, \ L \leq \omega \leq R \}.
\]
If $L \geq 0$ or $R \leq 0$, this set reduces to a simple convex set. So we assume $L<0$ and $R>0$. 
An integer formulation is:
\begin{equation}\label{MIP:absval}
\Jcal(L,R) := \left\{(\eta, \omega) \in [\ell, u]\times \Rbb \ \middle| \ 
\begin{bmatrix}\eta \\ \omega\end{bmatrix} = \begin{bmatrix} \lambda_1 |L| + \lambda_2 R \\  \lambda_1 L + \lambda_2 R
 \end{bmatrix}, \ \ 
\begin{aligned} 0&\leq \lambda_1 \leq 1- z, \\
0 &\leq \lambda_2 \leq z, 
\end{aligned} \ \ 
z \in \{0,1\}
\right\}.
\end{equation}
It is easy to verify that if $(\eta, \omega) \in \Jcal(L,R)$ and $z = 1$, then $\omega \in [0,R]$ and $\eta = \omega$, otherwise if $z = 0$ then $\omega \in [L,0]$ 
and $\eta = -\omega$. Therefore to derive integer formulations for $D_{\nu}(\ell,u)$ and $D^+_{\nu}(\ell,u)$ it suffices to 
derive finite bounds for the quantities inside the absolute value functions in (\ref{abs_cons}).
The following proposition and the subsequent remark provide the desired bounds.



\begin{proposition}
Let $x \in [\ell, u]$. 
Let $\{\xi_j(x), \eta_j(x)\}_{j=1}^\nu$ and $\{\omega_j(x)\}_{j=0}^{\nu-1}$  be parametric quantities defined as follows 
\begin{align*}
\begin{bmatrix} \xi_1(x) \\ \eta_1(x) \end{bmatrix} &:= \Fold\circ\Rot\left(\begin{bmatrix} x \\ x^2/2-0.5 \end{bmatrix}; \theta_{mid}\right), & \\ 
\omega_0(x) &:= -\sin(\theta_{mid}) \cdot x + \cos(\theta_{mid}) \cdot  (x^2/2-0.5) \\
\begin{bmatrix} \xi_{j+1}(x) \\ \eta_{j+1}(x) \end{bmatrix} &:= \Fold \circ \Rot\left(\begin{bmatrix} \xi_j(x) \\ \eta_j(x) \end{bmatrix}; \theta_d/(2^{j+1}) \right), 
& j=1,...,\nu - 1 \\
\omega_j(x) &:= -\sin(\theta_d/2^{j+1}) \cdot \xi_j(x) + \cos(\theta/2^{j+1}) \cdot \eta_j(x), &j=1,...,\nu-1.
\end{align*}
then 
\begin{align}
-(x^2/2+0.5) &\leq \omega_0(x) \leq x^2/2+0.5, \label{bound_omega:1} \\
-(x^2/2+0.5)\sin(\theta_d/2^{j+1}) &\leq \omega_j(x) \leq (x^2/2+0.5)\sin(\theta_d/2^{j+1}), \qquad \forall j \geq 1. \label{bound_omega:2}
\end{align}
\end{proposition}
\begin{proof}
By the invariance of 2-norm under the rotation and flipping operations, for $j=0,...,\nu$, the radial function
\[
\Rcal(x) := \sqrt{\xi_j^2(x) + \eta_j^2(x)} = \sqrt{\xi_0^2(x) + \eta_0^2(x)} = \sqrt{x^2 + (x^2/2-0.5)^2} = x^2/2+0.5.
\]
Therefore 
\[
\left| \omega_0(x) \right| \leq  \left\| \Rot([x, x^2/2-0.5]^T; \theta_{mid}) \right\| = \sqrt{x^2 + (x^2/2-0.5)^2} = x^2/2 + 0.5,
\]
and we have (\ref{bound_omega:1}).

It is easy to see that 
$\ArcTan\left(\Rot\left(\begin{bmatrix} x \\ x^2/2-0.5 \end{bmatrix}; \theta_{mid}\right)\right) \subseteq [ -\theta_d/2, \theta_d/2 ]+ \Zbb\cdot 2\pi$,
and for all $j\geq 1$, $\ArcTan\left(\Rot\left(\begin{bmatrix} \xi_j(x) \\ \eta_j(x) \end{bmatrix}; \theta_{d}/2^{j+1}\right)\right) \subseteq [ -\theta_d/2^{j+1}, \theta_d/2^{j+1} ]+ \Zbb\cdot 2\pi$. Since $\theta_d/2^{j+1} \in [0,\frac{\pi}{2}]$ for any $j\geq 1$ we have
\[
-\Rcal(x) \cdot \sin(\theta_d/2^{j+1}) \leq \omega_j(x) \leq \Rcal(x) \cdot \sin(\theta_d/2^{j+1}), \qquad \forall j \geq 1.
\]
This proves (\ref{bound_omega:2}).
\end{proof}

\begin{remark}
For $x \in [\ell, u]$,  $x^2/2+0.5$ can be upper bounded by $\max(\ell^2, u^2)/2 + 0.5$. Let us define constants $C$ and $C_j \ (j=1,...,\nu)$ by
\[
C = \max(\ell^2, u^2)/2 + 0.5, \qquad \mbox{ and }\ \  C_j := C \cdot \sin(\theta_d/2^{j+1}), \ \ \forall j=1,...,\nu.
\]
Then for any $(x,y) \in \Scal_{[\ell, u]}$, we have 
\begin{align*}
|-\sin(\theta_{mid}) x +  \cos(\theta_{mid})(y-1)/2| &\leq C \\ 
\left| -\sin(\theta_{d}/(2^{j+1}) \xi_j +  \cos(\theta_{d}/(2^{j+1})) \eta_j \right| &\leq C_j, \ \ \forall j=1,...,\nu-1.
\end{align*}
By using the idea of (\ref{MIP:absval}), we then construct integer formulation of constraints (\ref{def:xieta1}) and  (\ref{def:xieta2}) with additional variables 
$\{(\lambda_{j,1}, \lambda_{j,2}, z_j)\}_{j=1}^\nu$ and  some linear constraints:
\begin{align}
&\left\{
\begin{aligned}
\xi_1 & = x \cdot \cos(\theta_{mid}) + (y/2-0.5) \cdot \sin(\theta_{mid}), \\
(\lambda_{1,2}-\lambda_{1,1})\cdot C &= -x \cdot \sin(\theta_{mid}) + (y/2-0.5) \cdot \cos(\theta_{mid}), \\
\eta_1 & = (\lambda_{1,1}+\lambda_{1,2}) \cdot C, \\
0 &\leq \lambda_{1,1} \leq (1-z_1), \\
0 &\leq \lambda_{1,2} \leq z_1, \quad z_1 \in \{0,1\}.
\end{aligned}
\right.    \label{MIP:xieta1} \\
&\left\{
\begin{aligned}
\xi_{j+1} & = \xi_j \cdot \cos(\theta_d/2^{j+1}) + \eta_j \cdot \sin(\theta_d/2^{j+1}), \\
(\lambda_{j+1,2}-\lambda_{j+1,1})\cdot C_j &= - \xi_j \cdot \sin(\theta_d/2^{j+1}) + \eta_j \cdot \cos(\theta_d/2^{j+1}), \\
\eta_{j+1} & = (\lambda_{j+1,1}+\lambda_{j+1,2}) \cdot C_j, \\
0 &\leq \lambda_{j,1} \leq (1-z_{j+1}), \\
0 &\leq \lambda_{j,2} \leq z_{j+1}, \quad z_{j+1} \in \{0,1\}.
\end{aligned}
\right. \quad \forall j=1,...,\nu-1 \label{MIP:xieta2}
\end{align}
Therefore $D_\nu(\ell,u)$ and $D^+_\nu(\ell,u)$ have the following integer representations:
\begin{align}
D_\nu(\ell, u) & = \left\{ (x, y) \in B_{[\ell, u]} \ \middle| \ 
\begin{aligned}
\exists \{\xi_j, \eta_j, \lambda_{j,1}, \lambda_{j,2}, z_j\}_{j=1}^\nu,  \\
(\ref{MIP:xieta1}), \ (\ref{MIP:xieta2}), \ (\ref{ineq:secant}), \ (\ref{ineq:tangent1}), \ (\ref{ineq:tangent2})
\end{aligned}
\right\}, \label{MIP:D} \\
D^+_\nu(\ell, u) & = \left\{ (x, y) \in B_{[\ell, u]} \ \middle| \ 
\begin{aligned}
\exists \{\xi_j, \eta_j, \lambda_{j,1}, \lambda_{j,2}, z_j\}_{j=1}^\nu,  \\
(\ref{MIP:xieta1}), \ (\ref{MIP:xieta2}), \ (\ref{ineq:secant}), \ y \geq x^2
\end{aligned}
\right\}. \label{MIP:D+}
\end{align}
Note that these two integer formulations use at most $4\nu$ number of continuous variables and $\nu$ binary variables.
\end{remark}

\section{Disjunctive Characterization of the Approximation Sets}\label{sec:disjunctive}
In this section we seek to understand the geometry of approximation sets $D_\nu(\ell,u)$ and 
$D^+_\nu(\ell,u)$ in the original space, projecting out the additional lifted variables used in their constructions.
We show that both of the sets $D_\nu(\ell,u)$ and $D^+_\nu(\ell,u)$ are in fact the union of $2^\nu$ number of convex sets naturally 
approximating $\Scal_{[\ell,u]}$ in the original $(x,y)$ space. We first illustrate the this disjunctive characterization by Figure \ref{fig:disjunctive},
where the solid curves
depict set $\{(x, \frac{x^2-1}{2}) \ | \ x \in [\ell,u]\}$. This curve can be understood by first  embedding $\Scal_{[\ell,u]}$ into the boundary 
of the second-order-cone as in Figure \ref{fig:embedding} then projecting onto its first two dimensions. For fixed $\nu$, we  
partition the angular region $[\theta_{min}, \theta_{max}]$ into $2^\nu$ equi-angular sectors. Let $x_{1},...,x_{2^\nu+1}$ to denote the horizontal coordinates
of the intersection points between the solid curve and sector boundaries. Now consider set $\Scal_{[\ell,u]}$. Take all the supporting
tangent lines of $\Scal_{[\ell,u]}$ at each ``knot" $(x_j,x_j^2)$ ($\forall j$), and all secant lines passing two adjacent knots, we hence form a relaxation 
set of $\Scal_{[\ell,u]}$ which is the union of $2^\nu$ (convex) triangular regions. Results in this section show that this relaxation set coincides with $D_\nu(\ell,u)$.
Figure \ref{fig:disjunctive} illustrate the case of $\nu=1,2,3$, where the shaded region is $D_\nu(\ell,u)$ shifted/rescaled to match the curve $(x^2-1)/2$.
The disjunctive characterization of $D^+_\nu(\ell, u)$ is very much the same, except the lower piecewise linear boundary of $D_\nu(\ell,u)$ is replaced by
the smooth curve $y=x^2$ for $x \in [\ell,u]$.

\begin{figure}[htbp]
\begin{center}
\includegraphics[scale=0.6]{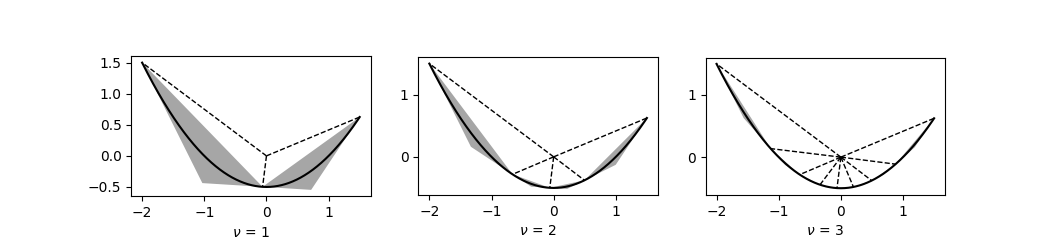}
\caption{Disjunctive interpretation of $D_\nu(\ell,u)$ with $\nu=1,2,3$, shaded region is $\{(x, (y-1)/2 \ | \ (x,y) \in D_\nu(\ell,u) \}$}
\label{fig:disjunctive}
\end{center}
\end{figure}

To algebraically prove this disjunctive characterization, we first establish an elementary lemma providing a characterization of one of such convex triangles.
\begin{lemma}\label{lem:Delta}
Let $\alpha$ and $\beta$ be two angles (in radians) with values in $\left(-\frac{3}{2}\pi, \frac{\pi}{2}\right)$. Let $x_\alpha$ be the unique value 
such that $\left(x_\alpha, \frac{x_\alpha^2-1}{2}\right) = (r \cos \alpha,  r\sin\alpha)$ for some $r>0$,  and $x_\beta$ be the unique value 
such that $\left(x_\beta, \frac{x_\beta^2-1}{2}\right) = (r' \cos \beta,  r'\sin\beta)$ for some $r' > 0$. 
 Then the triangular convex region, denoted as $\Delta_{\alpha,\beta}$, formed by the two tangent lines of 
 $\Ucal:= \{(x,x^2) \ | \ x \in \Rbb\}$ at $(x_\alpha, x_\alpha^2)$ and $(x_\beta, x_\beta^2)$, and the secant line connecting these two points, 
 is characterized by the following three inequalities:
\begin{equation} 
\Delta_{\alpha,\beta} = \left\{ (x,y) \ \middle| \
\begin{aligned}
x \cdot \cos\alpha + y \cdot \left(\frac{\sin\alpha-1}{2}\right) &\leq \frac{\sin\alpha+1}{2} \\
x \cdot \cos\beta + y \cdot \left(\frac{\sin\beta-1}{2}\right) &\leq \frac{\sin\beta+1}{2} \\
x  \cdot \cos\frac{\alpha+\beta}{2} + y \left(\frac{\sin\frac{\alpha+\beta}{2}-\cos\frac{\alpha-\beta}{2}}{2}\right) &\geq 
\frac{\sin\frac{\alpha+\beta}{2}+\cos\frac{\alpha-\beta}{2}}{2}
 \end{aligned} \right\}. \label{eq:Delta}
\end{equation} 
Furthermore, the convex region bounded by $\Ucal$ and the secant line passing through $(x_\alpha, x_\alpha^2)$ and $(x_\beta, x_\beta^2)$,
denoted by $\Delta^+_{\alpha,\beta}$, has the characterization
\begin{equation} 
\Delta^+_{\alpha,\beta} = \left\{ (x,y) \ \middle| \
\begin{aligned}
x  \cdot \cos\frac{\alpha+\beta}{2} + y \left(\frac{\sin\frac{\alpha+\beta}{2}-\cos\frac{\alpha-\beta}{2}}{2}\right) &\geq 
\frac{\sin\frac{\alpha+\beta}{2}+\cos\frac{\alpha-\beta}{2}}{2} \\
y &\geq x^2 
 \end{aligned} \right\}. \label{eq:Delta+}
\end{equation} 
\end{lemma}
\begin{proof}
Note that $x_\alpha$ is the $x$-coordinate of the unique intersection point of ray $\{(r \cos \alpha,  r\sin\alpha) \ | \ r > 0\}$
and $\left\{\left(x, \frac{x^2-1}{2}\right)\ \middle| \ x \in \Rbb \right\}$. 
We first claim that $x_\alpha =  \tan\alpha + \sec\alpha$. Note that if $\alpha=-\frac{\pi}{2}$ then $x_\alpha = 0$. Otherwise
\[
\tan\alpha = \frac{x_\alpha^2 - 1}{2x_\alpha} \ \ \Rightarrow \ \ x_\alpha = \tan\alpha + \sec\alpha \mbox{ or } \tan\alpha - \sec\alpha.
\]
If $\alpha \in (-3\pi/2, -\pi/2)$, by geometry we must have $x_\alpha < 0$. This rules out the possibility $\tan\alpha - \sec\alpha$
which is positive in this case. If $\alpha \in ( -\pi/2, \pi/2)$, $x_\alpha > 0$. Again we must have $x_\alpha = \tan\alpha + \sec\alpha$. 

It is then straightforward to compute the (supporting) tangent line of $\Ucal$ at $(x_\alpha, x_\alpha^2)$ is
\[
y \geq 2 (\tan\alpha + \sec\alpha) x - (\tan\alpha + \sec\alpha)^2.
\]
This is equivalent to the first inequality by multiplying with factor $\frac{1-\sin\alpha}{2}$ 
(which is nonzero for all $\alpha \in (-3\pi/2, \pi/2)$). 

Similarly, the second inequality in (\ref{eq:Delta}) characterizes the supporting tangent line of $\Ucal$ at $(x_\beta, x_\beta^2)$.

Finally, writing $x_\alpha = \tan\alpha +\sec\alpha = \frac{\sin\alpha+1}{\cos\alpha}$ and $x_\beta = \frac{\sin\beta+1}{\cos\beta}$, the secant inequality connecting $(x_\alpha, x_\alpha^2)$ and $(x_\beta, x_\beta^2)$ is
\[
y \leq \left(\frac{\sin\alpha+1}{\cos\alpha} + \frac{\sin\beta+1}{\cos\beta}\right) x - \frac{(\sin\alpha+1)(\sin\beta+1)}{\cos\alpha \cos\beta}.
\]
Multiplying $\frac{1}{2}\left(\cos \frac{\alpha-\beta}{2} - \sin\frac{\alpha+\beta}{2} \right)$ to both sides, to prove its equivalence to 
the third inequality in (\ref{eq:Delta}) it then suffices to verify that
\begin{align}
\left(\frac{\sin\alpha+1}{\cos\alpha} + \frac{\sin\beta+1}{\cos\beta}\right) \cdot \frac{1}{2}\left(\cos \frac{\alpha-\beta}{2} - \sin\frac{\alpha+\beta}{2}\right)
= \cos\frac{\alpha+\beta}{2}, \label{nasty-eq1} \\
\frac{(\sin\alpha+1)(\sin\beta+1)}{\cos\alpha \cos\beta} 
\frac{1}{2}\left(\cos \frac{\alpha-\beta}{2} - \sin\frac{\alpha+\beta}{2}\right) 
= \frac{\sin\frac{\alpha+\beta}{2}+\cos\frac{\alpha-\beta}{2}}{2}. \label{nasty-eq2}
\end{align}
(\ref{nasty-eq1}) can be verified by
\begin{align*}
LHS &= \frac{\sin(\alpha+\beta)+\cos\alpha + \cos\beta}{\cos\alpha  \cos\beta} \cdot \frac{1}{2}\left(\cos \frac{\alpha-\beta}{2} - \sin\frac{\alpha+\beta}{2}\right) \\
&= \frac{2 \sin\frac{\alpha+\beta}{2} \cos \frac{\alpha+\beta}{2} + 2 \cos \frac{\alpha+\beta}{2} \cos\frac{\alpha-\beta}{2}}{\cos\alpha  \cos\beta} \cdot \frac{1}{2}\left(\cos \frac{\alpha-\beta}{2} - \sin\frac{\alpha+\beta}{2}\right) \\
&= \frac{\cos\frac{\alpha+\beta}{2}\left[\cos^2\frac{\alpha-\beta}{2}-\sin^2\frac{\alpha+\beta}{2}\right] }{\cos\alpha \cos\beta} \\
&=  \frac{\cos\frac{\alpha+\beta}{2} \left[0.5\cos(\alpha-\beta)+0.5\cos(\alpha+\beta)\right] }{\cos\alpha \cos\beta} = \cos\frac{\alpha+\beta}{2}.
\end{align*}
To verify (\ref{nasty-eq2}), 
\begin{align*}
&\frac{(\sin\alpha+1)(\sin\beta+1)}{\cos\alpha \cos\beta} 
\frac{1}{2}\left(\cos \frac{\alpha-\beta}{2} - \sin\frac{\alpha+\beta}{2}\right)  \\
&= \frac{\left(2\sin(\alpha/2)\cos(\alpha/2)+1\right)\left(2\sin(\beta/2)\cos(\beta/2)+1\right)}{(\cos^2(\alpha/2)-\sin^2(\alpha/2))(\cos^2(\beta/2)-\sin^2(\beta/2))} \frac{1}{2} 
\left(\cos(\alpha/2) - \sin(\alpha/2)\right) \left(\cos(\beta/2) - \sin(\beta/2)\right) \\
&= \frac{\left(\sin(\alpha/2)+\cos(\alpha/2)\right)^2 \left(\sin(\beta/2)+\cos(\beta/2)\right)^2}{2\left(\sin(\alpha/2)+\cos(\alpha/2)\right) \left(\sin(\beta/2)+\cos(\beta/2)\right)} \\
&= \frac{1}{2}\left(\sin(\alpha/2)\cos(\beta/2)+\cos(\alpha/2)\sin(\beta/2)+\sin(\alpha/2)\sin(\beta/2)+\cos(\alpha/2)\cos(\beta/2)\right) \\
&= \frac{\sin((\alpha+\beta)/2) + \cos((\alpha-\beta)/2)}{2}.
\end{align*}
This completes our proof of (\ref{eq:Delta}). The proof of (\ref{eq:Delta+}) is analogous.
\end{proof}

We show that both $D_\nu(\ell,u)$ and $D_\nu^+(\ell,u)$ are the union of $2^\nu$ number of convex sets in the form 
of $\Delta_{\alpha,\beta}$ and $\Delta^+_{\alpha,\beta}$, respectively. In fact, we show that if we fix the binary $z$ vector 
in (\ref{MIP:D}) and  (\ref{MIP:D+}), then $D_\nu(\ell,u)$ and $D_\nu^+(\ell,u)$ reduces
to $\Delta_{\alpha,\beta}$ and $\Delta^+_{\alpha,\beta}$ with $\alpha,\beta$ being functions of $z$.

\begin{theorem}\label{thm:disjunctive}
Let $\ell \leq u$ and $\theta_{min}, \theta_d$ be angles as defined in (\ref{def:theta_minmax}) and (\ref{def:theta_mid}).
For fixed $z \in \{0,1\}^\nu$ and $(x,y) \in \Rbb^2$, 
there exists $\{\xi_j, \eta_j, \lambda_{j,1}, \lambda_{j,2}\}_{j=1}^\nu$
such that $\left(x,y, \{\xi_j, \eta_j, \lambda_{j,1}, \lambda_{j,2}, z_j\}_{j=1}^\nu \right)$ is feasible in (\ref{MIP:D}) if and only if 
\[
(x,y) \in \Delta_{\phi(z),\beta(z)},
\]
where $\Delta_{\bullet, \bullet}$ is the triangular convex set defined in Lemma \ref{lem:Delta}, and
\begin{equation}\label{def:phi_s}
\phi(z) := \theta_{min} + \sum_{i=0}^{\nu-1} \left[ (-1)^{\sum_{j=1}^i (1-z_j)}\right] \frac{\theta_d}{2^{i+1}}, \quad
\beta(z) := \phi(z) + (-1)^s \frac{\theta_d}{2^{\nu}}, \quad s := \sum_{j=1}^\nu (1-z_j).
\end{equation}
Furthermore, there exists $\{\xi_j, \eta_j, \lambda_{j,1}, \lambda_{j,2}\}_{j=1}^\nu$
such that $\left(x,y, \{\xi_j, \eta_j, \lambda_{j,1}, \lambda_{j,2}, z_j\}_{j=1}^\nu \right)$ is feasible in (\ref{MIP:D+}) if and only if 
\[
(x,y) \in \Delta^+_{\phi(z),\beta(z)}.
\]
\end{theorem}
\begin{proof}
Suppose that $\left(x,y, \{\xi_j, \eta_j, \lambda_{j,1}, \lambda_{j,2}, z_j\}_{j=1}^\nu \right)$ is feasible in (\ref{MIP:D}). 
Let $(r, \vartheta_{\kappa})$ to denote the polar coordinates of $(\xi_\kappa, \eta_\kappa)$ for $\kappa=1,2,...,\nu$, i.e., 
\[
\xi_\kappa = r \cos \vartheta_\kappa , \quad \eta_\kappa = r \sin \vartheta_\kappa, \quad  r \geq 0, \quad  \vartheta_\kappa \in [0, \theta_d/2^{\kappa}].
\]
Further let $(r, \vartheta_{0})$ to denote the polar coordinates of $(x, \frac{y-1}{2})$. Note the common radius $r$ is a consequence of the invariance of 2-norm
under folding and rotation operations. We claim that
\begin{equation}\label{eq:recurse_angle}
\vartheta_0 = \theta_{min} + \sum_{i=0}^{\kappa-1} \left[(-1)^{\sum_{j=1}^i (1-z_j)}\right]\frac{\theta_d}{2^{i+1}} + \left[(-1)^{\sum_{j=1}^\kappa (1-z_j)}\right]\vartheta_{\kappa},
\quad \kappa = 1,...,\nu,
\end{equation}
where $\sum_{j=1}^0 (1-z_j) = 0$ by convention. We prove this identity by induction. Note that 
$[\xi_{1}, \eta_{1}]$ is obtained by rotating $[x, (y-1)/2]$ clockwise by $\theta_{min}+\theta_{d}/2$
then folding up. By (\ref{MIP:xieta1}) it is easy to see that the folding operation has effects 
if and only if $z_{1}=0$. Reversing this procedure we have:
\[
\vartheta_{0} = \begin{cases}
\theta_{min} + \frac{\theta_d}{2} + \vartheta_{1}, & \mbox{ if } z_{1}=1, \\
\theta_{min} + \frac{\theta_d}{2} - \vartheta_{1}, & \mbox{ if } z_{1}= 0.
\end{cases}
\]
This is equivalent to
\[
\vartheta_{0} = \theta_{min} + \frac{\theta_d}{2} + (-1)^{1-z_{1}} \vartheta_{1}.
\]
Hence (\ref{eq:recurse_angle}) holds when $\kappa = 1$. Now assuming (\ref{eq:recurse_angle}) holds for $\kappa \leq \nu-1$, we show it is valid for 
$\kappa+1$.
The same geometric arguments establish the recursive identity:
\[
\vartheta_{\kappa} = \frac{\theta_d}{2^{\kappa+1}} + (-1)^{1-z_{\kappa+1}} \vartheta_{\kappa+1}.
\]
Therefore 
\begin{align*}
\vartheta_0 &= \theta_{min} + \sum_{i=0}^{\kappa-1} \left[(-1)^{\sum_{j=1}^i (1-z_j)}\right]\frac{\theta_d}{2^{i+1}} + \left[(-1)^{\sum_{j=1}^\kappa (1-z_j)}\right]\vartheta_{\kappa} \\
&= \theta_{min} + \sum_{i=0}^{\kappa-1} \left[(-1)^{\sum_{j=1}^i (1-z_j)}\right]\frac{\theta_d}{2^{i+1}} + \left[(-1)^{\sum_{j=1}^\kappa (1-z_j)}\right]\left[\frac{\theta_d}{2^{\kappa+1}} + (-1)^{1-z_{\kappa+1}} \vartheta_{\kappa+1}\right] \\
&=\theta_{min} + \sum_{i=0}^{\kappa} \left[(-1)^{\sum_{j=1}^i (1-z_j)}\right]\frac{\theta_d}{2^{i+1}} + \left[(-1)^{\sum_{j=1}^{\kappa+1} (1-z_j)}\right]\vartheta_{\kappa+1},
\end{align*}
which proves (\ref{eq:recurse_angle}). Take $\kappa =\nu$, let $\phi(z)$ and $s$ be quantities as defined in (\ref{def:phi_s}), we have
\[
\theta_0 = \phi(z) + (-1)^s\theta_\nu \quad \Rightarrow \quad \theta_\nu = (-1)^s \left[\vartheta_0 - \phi_\nu(z) \right].
\]

Converting to the rectangular coordinates we have
\begin{align}
\xi_\nu = r\cos (\vartheta_\nu) &= \cos(\phi_\nu(z)) r\cos(\vartheta_0) +  \sin(\phi_\nu(z)) r\sin(\vartheta_0) \notag \\
&= \cos(\phi_\nu(z)) x + \sin(\phi_\nu(z))  (y-1)/2, \label{eq:xi_with_xy} \\
\eta_\nu = r\sin (\vartheta_\nu) &=  (-1)^s \left[r\cos(\phi_\nu(z)) \sin(\vartheta_0) - r \sin(\phi_\nu(z)) \cos(\vartheta_0)\right] \notag \\
&=  (-1)^s\left[\cos(\phi_\nu(z))(y-1)/2 - \sin(\phi_\nu(z)) x\right]. \label{eq:eta_with_xy}
\end{align}

Now substituting $(\xi_\nu, \eta_\nu)$ in (\ref{ineq:secant} -- \ref{ineq:tangent2}) with (\ref{eq:xi_with_xy}) and (\ref{eq:eta_with_xy}), with straightforward computation
we obtain the following three inequalities
\begin{align*}
x \cos\left(\phi(z) + (-1)^s\frac{\theta_d}{2^{\nu+1}}\right) + y\left[\frac{\sin\left(\phi(z) + (-1)^s\frac{\theta_d}{2^{\nu+1}}\right)}{2}\right]
\geq \frac{1}{2} \left[ \cos\left(\frac{\theta_d}{2^{\nu+1}}\right) + \sin\left(\phi(z) + (-1)^s\frac{\theta_d}{2^{\nu+1}}\right) \right] \\
x \cos \left(\phi_\nu(z) +(-1)^s\frac{\theta_d}{2^{\nu}} \right) + y \left[\frac{\sin\left(\phi_\nu(z) +(-1)^s\frac{\theta_d}{2^{\nu}} \right)-1}{2}\right]
\leq \frac{1+\sin\left(\phi_\nu(z) +(-1)^s\frac{\theta_d}{2^{\nu}} \right)}{2} \\
x \cos(\phi(z)) + y\left(\frac{\sin(\phi(z))-1}{2}\right) \leq \frac{\sin(\phi(z))+1}{2}.
\end{align*}
By Lemma \ref{lem:Delta} these inequalities exactly characterize $\Delta_{\phi(z), \beta(z)}$ where $\beta(z) := \phi(z) + (-1)^s \frac{\theta_d}{2^{\nu}}$.
In other words if $\left(x,y, \{\xi_j, \eta_j, \lambda_{j,1}, \lambda_{j,2}, z_j\}_{j=1}^\nu \right)$ is feasible in (\ref{MIP:D}) then 
$(x,y) \in \Delta_{\phi(z), \beta(z)}$. 

To prove the converse, note that if $(x,y) \in \Delta_{\phi(z), \beta(z)}$ for some $z \in \{0,1\}^\nu$, then all auxiliary variables 
$\{\xi_j, \eta_j, \lambda_{j,1}, \lambda_{j,2}\}_{j=1}^\nu$ are entirely determined by the recursive identity (\ref{eq:recurse_angle}). Our 
proof is then straightforward by recognizing the equivalence of  (\ref{ineq:secant} -- \ref{ineq:tangent2}) with the three inequalities
in (\ref{eq:Delta}). The proof for the statement on $D^+_\nu(\ell,u)$ and $\Delta^+_{\phi(z),\beta(z)}$ is entirely analogous.
%
%
%
\end{proof}

\begin{remark} Theorem \ref{thm:disjunctive} suggests the following disjunctive characterization of $D_\nu(\ell,u)$ and $D^+_\nu(\ell,u)$:
\begin{align*}
D_\nu(\ell,u) = \bigcup_{z \in \{0,1\}^\nu} \Delta_{\phi(z), \beta(z)},\\
D^+_\nu(\ell,u) = \bigcup_{z \in \{0,1\}^\nu} \Delta^+_{\phi(z), \beta(z)},
\end{align*}
where $\phi(z)$ and $\beta(z)$ are as defined in (\ref{def:phi_s}). This characterization may remind readers some recent work 
in constructing mixed-integer representations for the union of finitely number of polyhedral sets, e.g., \cite{VielmaNemhauser2011,VielmaSIAM2015,HuchetteVielma}.
In these works, methods were presented for constructing mixed-integer formulations with a \textit{logarithmic} (to the number of original polyhedral sets) number of 
\textit{auxiliary} binary and continuous variables.
However as their approach applies to fairly general kind of polyhedral sets, it requires at least \textit{enumerating} the structure (extreme points, etc.) of all polyhedral 
set. For example consider the notion of \textit{combinatorial disjunctive constraint} (CDC) defined in \cite{HuchetteVielma} to generalize the approach of \cite{VielmaNemhauser2011}.
Let $J = \bigcup_{i=1}^d S_i$ where $S_i:=\ext(P^i)$ is the set of extreme points of the $i$-th polytope $P^i$. To represent a (continuous) optimization
variable in the union of all $P^i$, without assuming any relations among the sets $\{P^i\}_{i=1}^d$, it is 
necessary to introduce a variable $\lambda_v (\geq 0)$ for each $v \in J$, such that $\sum_{v \in J} \lambda_v = 1$ and use the combination
\[
\sum_{v \in J}\lambda_v v.
\]
In our setting (take $D_\nu(\ell,u)$ for example), the number of polyhedral sets is $d = 2^\nu$ and the number of extreme points is already $3\cdot 2^\nu$. 
So by simply applying their approach it is unlikely to obtain an mixed-integer formulation for $D_\nu(\ell,u)$ with a \textit{totally} $O(\nu)$ number of variables 
and constraints. Our approach achieve this rate by exploiting the implicit symmetry of the square function, or in other words, the symmetry among the 
specifically constructed sets $\left\{\Delta_{\phi(z), \beta(z)}\right\}_{z\in \{0,1\}^\nu}$ or $\left\{\Delta^+_{\phi(z), \beta(z)}\right\}_{z\in \{0,1\}^\nu}$.
\end{remark}

%

\section{An adaptive refinement algorithm}\label{sec:adaptive_refine}
In this and the next sections  
we describe an iterative refinement algorithm and use it to evaluate whether our proposed approach is promising. 
Since integer variable $\zeta$ in (\ref{MIQCP}) plays no role in our approximation approach,
we focus on nonconvex quadratically constrained program
 with continuous variables:
 \begin{equation}\label{QCP} 
\begin{aligned}
\tau := \min_{x \in \Rbb^n} & \quad x^T Q^{(0)} x + c^{(0)T} x \\
s.t.& \quad x^T \matQi x + \vecciT x \leq f^{(i)}, \quad i=1,...,m, \\
& \quad x_j \in [\ell_j, u_j], \ \forall j = 1,...,n. \\
\end{aligned}\tag{QCP}
\end{equation}
We leave the development of a full fledged MIQCP solver in subsequent works. See Section \ref{sec:conclude} for some further 
discussion and considerations.



In the preprocessing step, for any $Q^{(i)}$ not positive semidefinite we solve the following auxiliary semidefinite program (SDP)
\begin{equation}\label{SDP:DiagPerturb}
\min_{\delta^{(i)} \in \Rbb^n} \quad e^T \delta^{(i)} \quad s.t., \quad \matQi + \diag(\delta^{(i)}) \succeq 0, \tag{SDP}
\end{equation}
where $e$ is the vector of ones in $\Rbb^n$. This is an SDP with a very special form. It is the same as the dual of the well-known semidefinite relaxations to the Max-Cut problem \cite{GoWi95}. 
This problem structure can be effectively exploited by the solver DSDP \cite{DSDPAlg,dsdp-user-guide}. Another attractive
method (as solution accuracy is not a primary concern here) is a coordinate-minimization-based algorithm
proposed in \cite{DongNonconvexQP} where a simple rank-1 update operation is needed in each iteration.

One special case is that if $Q^{(i)}$ is diagonal (but nonconvex), then we can simply choose $\delta^{(i)}$ such that $Q^{(i)} + \diag(\delta^{(i)})=0$.

We use $\nu_j$ to denote the level of approximation to the set $\Scal_{[\ell_j, u_j]}$.  By convention
we define $D_0(\ell, u)$ and $D^+_0(\ell, u)$ as the following initial convex relaxation of $\Scal_{[\ell,u]}$:
\begin{align*}
D_0(\ell,u) &:= \left\{(x,y) \in \Rbb \times \Rbb_+ \ \middle| \ \begin{aligned} x \in [\ell,u],  \ \ y \leq (\ell+u) x - \ell u \\
y \geq 2 \ell x - \ell^2, \ \ y \geq 2 u x - u^2\\
\end{aligned}
\right\}, \\
D^+_0(\ell,u) &:= \left\{(x,y) \in \Rbb^2 \ \middle| \ \begin{aligned} x \in [\ell,u],  \ \ y \geq x^2 \\
y \leq (\ell+u) x - \ell u \\
\end{aligned}
\right\}.
\end{align*}
Starting with $\nu_j = 0 \ (j=1,...,n)$, in each iteration we solve the following (mixed-integer) convex quadratic program 
to globally optimality and increase $\nu_j$ for those variables determined to be ``important" (explained later). 
\begin{equation}\label{Relax}
\begin{aligned}
\min_{x \in \Rbb^n} & \quad  x^T \left(Q^{(0)} + \diag(\delta^{(0)}) \right) x + c^{(0)T} x - \delta^{(0)T} y \\
s.t.& \quad x^T \left(\matQi + \diag(\delta^{(i)})\right) x + \vecciT x \leq \scldi + \delta^{(i)T} y, \quad i=1,...,m, \\
& \quad (x_j, y_j) \in D_{\nu_j}(\ell_j, u_j) \quad (\mbox{ or } D^+_{\nu_j}(\ell_j, u_j) \ ).
\end{aligned}\tag{R}
\end{equation}
Note the optimal value to (\ref{Relax}) is always a valid lower bound of $\tau$, the optimal value of (\ref{QCP}), we use 
$\tau_{lower}$ to denote the best lower bound obtained so far.
Let $\left\{x^{(k)}_j, y^{(k)}_j\right\}_{j=1}^n$ to denote the optimal solution to (\ref{Relax}) found in iteration $k$. 
We then solve (\ref{QCP}) locally by using a nonlinear optimization algorithm with initial point $\{x_j^{(k)}\}_{j=1}^n$
to obtain a (hopefully) feasible solution to (\ref{QCP}). If feasible, the corresponding objective value is used to update $\tau_{upper}$, the best upper bound of $\tau$ so far.
We then increase $\nu_j$ (by 1) for the $T$ number of indices with largest violation scores $|y_j - x_j^2|$, unless
the corresponding violation score falls below some threshold $\varepsilon_{viol}$. 
We terminate our algorithm if $\tau_{upper}$ and $\tau_{lower}$ are sufficiently close or no $\nu_j$ is increased.
The pseudo code for this algorithm is presented in Algorithm \ref{alg:adref}.
\begin{algorithm}[ht]
\caption{An adaptive refinement algorithm for (\ref{QCP})}
\label{alg:adref}
\begin{algorithmic}
    \Require{Problem data $\left\{Q^{i}, c^{(i)}\right\}_{i=0}^m, \left\{f^{i}\right\}_{j=1}^m,$ and $\{\ell_j, u_j\}_{j=1}^n$}. Parameters 
    $T$, $\varepsilon_{viol}$, $\varepsilon_{gap}$.
    \State Preprocessing: Use a nonlinear optimization algorithm to solve (\ref{QCP}) locally; if a feasible solution is
	found, set $\tau_{upper}$ to be the corresponding objective value; otherwise $\tau_{upper} \leftarrow +\infty$;
                $\tau_{lower} \leftarrow -\infty$; \ $\nu_j \leftarrow 0$ for all $j=1,...,n$; 
    \For{$i=1,...,m$}
		\If{$Q^{(i)}$ is diagonal}
			$\delta^{(i)} \leftarrow -\diag(Q^{(i)})$;
		\ElsIf{$Q^{(i)}$ is positive semidefinite}
			$\delta^{(i)} \leftarrow 0 $;
		\Else
		\State Solve (\ref{SDP:DiagPerturb}) to find $\delta^{(i)}$;
		\EndIf
	\EndFor
    \For{$k=1,2,...$}
	\State Solve (\ref{Relax}) to global optimality; let $\{x_j^{(k)}, y_j^{(k)}\}_{j=1}^n$ to denote an optimal solution; and $\tau^{(k)}$
	to denote the optimal value of (\ref{Relax});
	\State $\tau_{lower} \leftarrow \max(\tau_{lower}, \tau^{(k)})$;
	\State Use a nonlinear optimization algorithm to solve (\ref{QCP}) locally with initial point $\{x_j^{(k)}\}_{j=1}^n$; if a feasible solution is
	found, let $\tau^{(k)}_{f}$	to denote the corresponding objective value;
	\State $\tau_{upper} \leftarrow \min(\tau_{upper}, \tau_f^{(k)})$;
	\State Terminate if $|\tau_{upper}-\tau_{lower}|/|\tau_{upper}| \leq \varepsilon_{opt}$;
	\State Let $j_1, ..., j_T$ be the indices for the $T$ largest entries in $\{ |y_j - x_j^2| \}_{j=1}^n$;
	\For{$j=j_1,j_2,...,j_T$}
		\If{$|y_j - x_j^2|>\varepsilon_{viol}$}
		\State $\nu_j \leftarrow \nu_j + 1$;
		\EndIf
	\EndFor
	\State Terminate if no $\nu_j$ is updated.
	\EndFor
\end{algorithmic}
\end{algorithm}

Note that either $D_{\nu_j}(\ell_j, u_j)$ or $D^+_{\nu_j}(\ell_j, u_j)$ can be used in the relaxation problem (\ref{Relax}). In our implementation,
if all $Q^{(i)}$ in the constraints of (\ref{QCP}) are either zero or diagonal, then we use $D_{\nu_j}(\ell_j, u_j)$ in (\ref{Relax}). In this case all constraints
in (\ref{Relax}) are linear (so that convex QP relaxations are solved at each node of the branch-and-bound). In all other cases we use $D^+_{\nu_j}(\ell_j, u_j)$ in (\ref{Relax}). We implement this algorithm with 
Julia + JuMP \cite{LubinDunningIJOC}, and use DSDP \cite{DSDPAlg,dsdp-user-guide} and Gurobi \cite{gurobi2018} to solve (\ref{SDP:DiagPerturb}) and (\ref{Relax}), 
respectively.
In all experiments described in the next section we 
set the parameters $T = 20$, $\varepsilon_{viol} = 10^{-5}$ and $\varepsilon_{gap} = 10^{-4}$.

\section{Computational Experiments}\label{sec:comp}

We now report our numerical results on 234 problem instances from the literature: 
the \textsc{boxqp} instances in \cite{BurerChen2013} and earlier papers, and \textsc{qcqp} instances 
included in the BARON test library \cite{BaronTestLibraryUrl}. 
Characteristics of such instances are summarized in the Table \ref{tab:instances}, where the ``Sparsity" column
represents the sparsity level in all quadratic forms. 
These
instances are sufficiently nonconvex, nonlinear and all bounded in a proper scale, therefore a good testbed 
for our purpose of evaluating whether our proposed approach is promising for dealing with challenging MIQCP problems. 
\begin{table}[htp]
\begin{center}
\begin{tabular}{|l|cccc|}
& \# instances & \#Vars & \#QuadCons & Sparsity \\
\hline
\textsc{boxqp} & 99 & 20--125 & 0 & 20\%--100\%\\
\textsc{qcqp}& 135 & 8--50 & 8--100 & 25\%--100\%\\
\hline
\end{tabular}
\end{center}
\caption{Characteristics of Test Instances}
\label{tab:instances}
\end{table}%

In all experiments below, if an algorithm cannot solve an instance to global optimality within the time limit, 
the corresponding relative gap is calculated as
\[
Gap = \frac{UpperBound - LowerBound}{|UpperBound|} \times 100\%.
\]
If both algorithms under comparison cannot solve an instance to global optimality within the time limit, we 
compare their lower bounds by the following percentage of 
``additional gap closed":
\[
AdditonalGapClosed = \frac{BetterLowerBound-WorseLowerBound}{BestUpperBound - WorseLowerBound} \times 100\%.
\]

As there is no quadratic constraint in the \textsc{boxqp} instances, they can be solved by 
Cplex (with ``Optimality Target" option set to be global) \cite{cplex12_7}. All other Cplex 
options are set to default. Figure \ref{fig:boxqp_res} summarizes the timing and gap 
comparison. We ran both Cplex and our adaptive refinement implementation (denoted by
``CDA") with a time limit of 1200 seconds. 

Among 99 \textsc{boxqp} instances, 81 of them can be solved to global optimality
(relative gap less than $10^{-4}$) by at least one method within the time limit. 
Figure \ref{fig:boxqpsub1} plots the Cplex running
time against CDA running time on such ``easier" instances. Any point in the region \textit{below} the 
diagonal straight line represents an instance where CDA is faster. Conclusion is that among
these easier instances, the running time for two methods are comparable. Although on the 
``easiest" instances (a cluster at the lower left corner), Cplex is faster. This is expected 
as our implementation is prototypical, and every (convex MIQCP) subproblem in the adaptive 
refinement procedure is resolved from the scratch. 

The advantage of CDA becomes apparent on more ``difficult" problems. Cplex leaves positive ($>10^{-4}$) gap 
on 24 instances with an average gap $13.14\%$, while CDA on 
23 instances with an average gap $1.20\%$.
On 18 instances which neither algorithm completes within the time limit, Cplex returns smaller gap on only 1 
instance ($\mbox{CplexGap}=2.54\%$ v.s. $\mbox{CDAGap}=2.56\%$). On all other 17 instances
(Figure \ref{fig:boxqpsub2}), 
CDA closes significantly portions of the gap left by Cplex ($\approx 90\%$ in average). 
\begin{figure}[htbp]
\begin{center}
\begin{subfigure}{.45\textwidth}
  \centering
  \includegraphics[width=\linewidth]{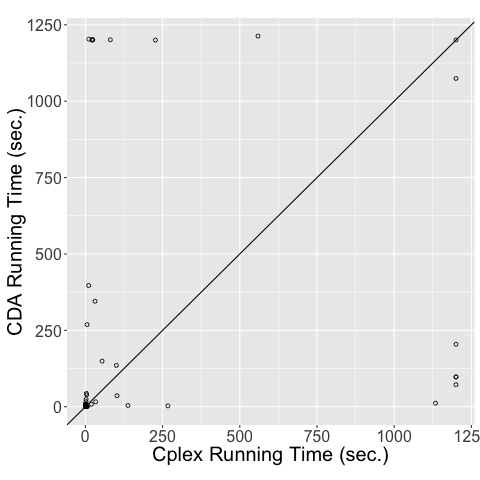}
  \caption{81 instances solved by at least one method.}
  \label{fig:boxqpsub1}
\end{subfigure}%
\begin{subfigure}{.45\textwidth}
  \centering
  \includegraphics[width=\linewidth]{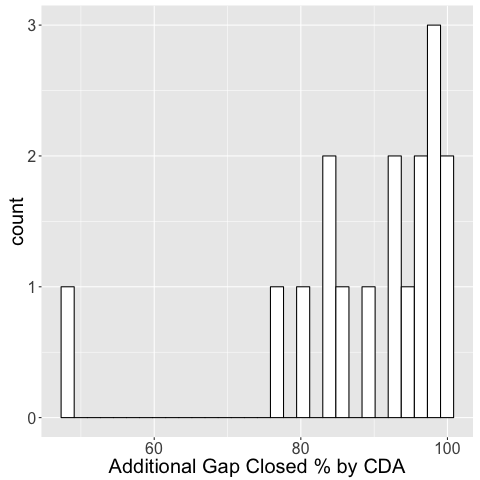}
  \caption{17 instances solved by neither method.}
  \label{fig:boxqpsub2}
\end{subfigure}
\caption{\textsc{boxqp} results summary}
\label{fig:boxqp_res}
\end{center}
\end{figure}

\begin{figure}[htbp]
\begin{center}
\begin{subfigure}{.45\textwidth}
  \centering
  \includegraphics[width=\linewidth]{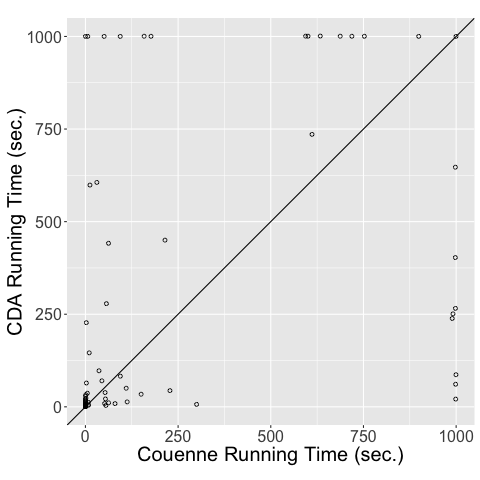}
  \caption{94 instances solved by at least one method.}
  \label{fig:qcqp_sub1}
\end{subfigure}%
\begin{subfigure}{.45\textwidth}
  \centering
  \includegraphics[width=\linewidth]{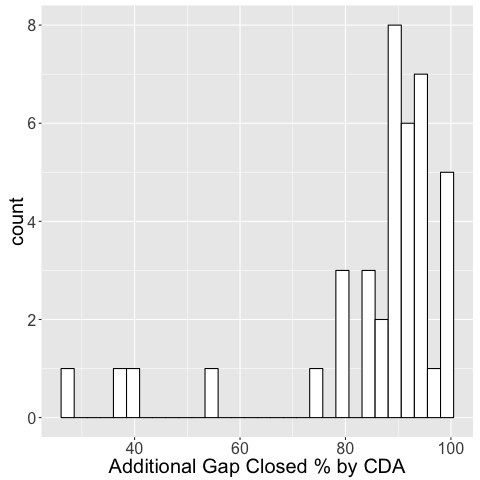}
  \caption{40 instances solved by neither method.}
  \label{fig:qcqp_sub2}
\end{subfigure}
\caption{\textsc{qcqp} results summary}
\label{fig:qcqp_res}
\end{center}
\end{figure}

Figure \ref{fig:qcqp_res} summarizes similar experiments of running 
CDA and Couenne 0.5 \cite{couenne} on the \textsc{qcqp} instances with every similar
conclusions. Among 135 \textsc{qcqp} instances, 81 of them were solved to sufficient accuracy by at least one 
method within the time limit (1000 seconds). Figure \ref{fig:qcqp_sub1} plots the Couenne 
running
time against CDA running time, and any point in the region \textit{below} the 
diagonal line represents an instance CDA is faster. Again the timing comparison 
on such instances is mixed, although on the 
``easiest" instances (a cluster at the lower left corner) Couenne is faster, which is again expected. 
Couenne leaves positive ($>10^{-4}$) gap 
on 49 instances with an average gap $60.39\%$, while CDA on 
55 instances with an average gap $5.24\%$.
On 41 instances which neither algorithm completes within the time limit, Couenne returns smaller gap on only 1 
instance ($\mbox{CouenneGap}=2.39\%$ v.s. $\mbox{CDAGap}=12.78\%$). On all other 40 instances
(Figure \ref{fig:qcqp_sub2}), 
CDA closes significantly portions of the gap left by Couenne (more than $80\%$ in average). 

To compare with commercial global optimization softwares such as BARON and ANTIGONE, we manually 
upload some larger \textsc{qcqp} instances to the NEOS server (\url{https://neos-server.org/}) and solve with a time limit of 1000 
seconds. The obtained lower bounds and the best upper bound (``BestObj" column) are listed in  Table \ref{tab:baron_antigone}.
While BARON and ANTIGONE typically provide much tighter bounds than Couenne 0.5, CDA still provides the strongest lower bounds
on all but one instance. The improvement is especially significant on larger instances (with 50 variables).
\begin{table}[htp]
\begin{center}
\begin{tabular}{lr|r|r|r|r|r|}
& & \multicolumn{4}{c|}{Lower Bounds} \\
Instance & BestObj & Couenne & BARON & ANTIGONE & CDA \\
\hline
unitbox\_c\_40\_80\_3\_100   & -84.084 & -175.238 & -92.790 & -93.169&  $-92.768^*$\\
unitbox\_c\_40\_80\_3\_50   & -49.471 & -81.459 &  -71.538& -56.350 & $-56.040^*$ \\
unitbox\_c\_48\_96\_2\_25   & -38.414 & -53.616 & -47.431 &$-42.176^*$ & -45.145 \\
unitbox\_c\_50\_50\_1\_100   & -101.038 & -283.876 & -129.895 & -131.035 & $-109.927^*$ \\
unitbox\_c\_50\_50\_1\_50    &-57.959&-123.338&-104.537& -79.899 & $-65.282^*$ \\
unitbox\_c\_50\_50\_2\_100   & -77.774 & -263.097 & -104.130 & -112.472&$-88.448^*$ \\
unitbox\_c\_50\_100\_1\_100   & -95.198  & -281.132 & -117.214&-124.912 &$-105.009^*$ \\
unitbox\_c\_50\_100\_1\_50   & -86.392  & -148.437 &-120.065 & -105.205& $-96.378^*$\\
\hline
\end{tabular}
\end{center}
\caption{BARON/ANTIGONE lower bounds on large \textsc{qcqp} instances (* - best lower bounds)}
\label{tab:baron_antigone}
\end{table}%

\section{Conclusions and Future Work}\label{sec:conclude}
In this paper we proposed a new approach, based on novel compact lifted mixed-integer approximation sets to
the bounded square function, to solve nonconvex MIQCP to arbitrary precision with MISOCP solvers.
This approach exploits efficient softwares and algorithmic progresses in MISOCP. 
The compact approximation sets are
derived by embedding the square function into a second-order cone and exploiting rotational symmetry.
We further characterize approximation precisions and provide disjunctive interpretations without lifted variables.
Finally, we implement a prototypical adaptive refinement algorithm for continuous QCPs. 
Preliminary numerical
experiments show that our implementation can close a significant portion
of gap left by state-of-the-art global solvers on more difficult problems, indicating 
strong promises of our proposed approach.

The adaptive refinement algorithm described in Section \ref{sec:adaptive_refine}
serves the purpose of prototyping, 
and may not be the most efficient way of employing the proposed approximation sets $D_\nu$
and $D_\nu^+$. For one thing, MISOCPs in iterations are solved from the scratch as it is not possible
to reuse detailed information of previous branch-and-bound trees. A potentially better implementation strategy
is to start with sufficiently accurate approximations and use specialized branching rules. 
Incorporation of processing techniques such as (feasibility and optimality-based) bound tightening 
is necessary for a robust algorithm for general MIQCPs.
Other interesting questions include how to incorporate the power of RLT inequalities in our framework
and how to (most effectively) generalize to mixed-integer polynomial optimization. 
We plan to address these questions and to develop a full fledged solvers in future work.

\bibliographystyle{plain}
\bibliography{miqcp}
\end{document}